\documentclass{birkjour_t2}
\usepackage{amsmath}
\usepackage{amsfonts}
\usepackage{amssymb}
\usepackage[english]{babel}
\usepackage{color}
\usepackage{graphicx}
\usepackage{tkz-euclide}
\usetkzobj{all}
\tikzset{elegant/.style={smooth,thick,samples=50,cyan}}
\tikzset{eaxis/.style={->,>=stealth}}
\usetikzlibrary{decorations.pathreplacing}
\usepackage{lmodern}
\usepackage{amsthm}
\usepackage{graphicx}
\usepackage{mathrsfs}
\usepackage{makecell}
\usepackage{microtype}
\usepackage{enumerate}
\usepackage{mathscinet}
\usepackage{array}
\usepackage{multirow}
\usepackage{enumerate}
\usepackage[cal=boondoxo,bb=ams]{mathalfa}
\usepackage{hyperref}
\hypersetup{hidelinks}

\renewcommand{\theequation}{\arabic{section}.\arabic{equation}}

\newtheorem{theorem}{Theorem}[section]
\newtheorem{prop}{Proposition}[section]
\newtheorem{lemma}{Lemma}[section]

\newtheorem{remark}{Remark}[section]

\newcommand{\ml}{\mathcal}
\newcommand{\mb}{\mathbb}
\DeclareMathOperator{\diag}{diag}

\DeclareMathOperator{\intt}{int}
\DeclareMathOperator{\extt}{ext}
\DeclareMathOperator{\midd}{mid}

\begin{document}

%
%
%
%
%
%
%
%

\title[Generalized thermoelastic plate equations]
 {Asymptotic profiles of solutions for regularity-loss type generalized thermoelastic plate equations and their applications}

\author[Y. Liu]{Yan Liu}
\address{Department of Applied Mathematics\\ Guangdong University of Finance\\
	 	510521 Guangzhou\\
	 	 P.R. China}
\email{liuyan99021324@tom.com}
\author[W. Chen]{Wenhui Chen$^*$}
\address{Institute of Applied Analysis, Faculty for Mathematics and Computer Science\\
	Technical University Bergakademie Freiberg\\
	Pr\"{u}ferstra{\ss}e 9\\
	09596 Freiberg\\
	Germany}
\email{wenhui.chen.math@gmail.com}


\date{March 15, 2019}

\begin{abstract}
In this paper, we consider generalized thermoelastic plate equations with Fourier's law of heat conduction. By introducing a threshold for decay properties of regularity-loss, we investigate decay estimates of solutions with/without regularity-loss in a framework of weighted $L^1$ spaces.  Furthermore, asymptotic profiles of solutions are obtained by using representations of solutions in the Fourier space, which are derived by employing WKB analysis. Next, we study generalized thermoelastic plate equations with additional structural damping, and analysis the influence of structural damping on decay properties and asymptotic profiles of solutions. We find that the regularity-loss structure is destroyed by structural damping. Finally, we give some applications of our results on thermoelastic plate equations and damped Moore-Gibson-Thompson equation.
\end{abstract}

\maketitle

\section{Introduction}\label{Sec.Intro.}
In recent years, the Cauchy problems for evolution-parabolic coupled systems have caught a lot of attention. Particularly, the so-called $\alpha-\beta$ system describes a few physical models, including second-order thermoelastic, thermoelastic plate equations, and linear viscoelastic equation. The Cauchy problem for the abstract form of $\alpha-\beta$ system can be modeled by
\begin{equation}\label{Eq.Alpha-Beta}
\left\{
\begin{aligned}
&u_{tt}+\ml{A}u-\gamma_1\ml{A}^{\alpha}v=0,&&x\in\mb{R}^n,\,\,t>0,\\
&v_t+\gamma_2\ml{A}^{\beta}v+\gamma_1\ml{A}^{\alpha}u_t=0,&&x\in\mb{R}^n,\,\,t>0,\\
&(u,u_t,v)(0,x)=(u_0,u_1,v_0)(x),&&x\in\mb{R}^n,
\end{aligned}
\right.
\end{equation}
where $(\alpha,\beta)\in[0,1]\times[0,1]$, $\gamma_1\in\mb{R}\backslash\{0\}$, $\gamma_2\in\mb{R}_+$, and $\ml{A}$ denotes a self-adjoint operator on a Hilbert space. The regularity analysis for \eqref{Eq.Alpha-Beta} have been developed in \cite{ChenRussell1981,ChenTriggiani1989,ChenTriggiani1990,DenkRacke,HaoLiu2013,HaoLiu2015,Huang1988,HuangLiu1988,LiuYong1998,MunozRiveraRacke1996}, where the semigroup associated with the system is analytic, of specific order Gevrey classes, and non-smoothing have been obtained. Moreover, $L^p$-resolvent estimates and time decay estimates for the solutions in the $L^q$ norm with $q\in[2,\infty]$ have been derived \cite{DenkRacke}. 

To understand the asymptotic properties of solutions to \eqref{Eq.Alpha-Beta}, we always need to describe long-time behavior of solutions. Especially, the properties of diffusion phenomena and asymptotic profiles of solutions provide some opportunities for us to get a long-time approximation of solutions. For the case when $\ml{A}=-\partial_x^2$, $\alpha=1/2$ and $\beta=1$ in \eqref{Eq.Alpha-Beta}, the thesis \cite{Jachmann2008} obtained diffusion phenomenon for second-order thermoelastic in 1D. Recently, concerning $\ml{A}=-\Delta$, $\alpha=0$ and $\beta=1$ in \eqref{Eq.Alpha-Beta}, the author of \cite{Ueda2018} derived sharp energy estimates and asymptotic profiles of solutions for the hyperbolic-parabolic coupled system with constraint condition on initial data. However, to the best of the authors' knowledge, asymptotic profiles of solutions to the general $\alpha-\beta$ system \eqref{Eq.Alpha-Beta} are still unknown. In this paper, we will give the answers for the case when $\alpha=\beta$.

In this paper, we first consider $\alpha-\beta$ system \eqref{Eq.Alpha-Beta} with $\alpha=\beta$, where we may assume $\gamma_1=\gamma_2=1$ without loss of generality. More precisely, we study the following Cauchy problem for generalized thermoelastic plate equations:
\begin{equation}\label{Eq.Gen.Ther.Plate.Eq.}
\left\{
\begin{aligned}
&u_{tt}+\ml{A}u-\ml{A}^{\alpha}v=0,&&x\in\mb{R}^n,\,\,t>0,\\
&v_t+\ml{A}^{\alpha}v+\ml{A}^{\alpha}u_t=0,&&x\in\mb{R}^n,\,\,t>0,\\
&(u,u_t,v)(0,x)=(u_0,u_1,v_0)(x),&&x\in\mb{R}^n,
\end{aligned}
\right.
\end{equation}
where $\alpha\in[0,1]$, and $\ml{A}=(-\Delta)^{\sigma}$ with $\sigma\in[1,\infty)$. Here the fractional power operator
\begin{align*}
\ml{A}^{\alpha}:\,\,\ml{D}(\ml{A}^{\alpha})\subset L^2\rightarrow L^2,\quad\text{where }\alpha\in[0,1],
\end{align*}
with its domain $\ml{D}(\ml{A}^{\alpha})=H^{2\sigma\alpha}$ and range $\ml{R}(\ml{A}^{\alpha})\subset L^2$ is defined by
\begin{align*}
(\ml{A}^{\alpha}f)(x):=\ml{F}^{-1}\left(|\xi|^{2\sigma\alpha}\ml{F}(f)(\xi)\right)(x),
\end{align*}
where $f\in H^{2\sigma\alpha}$. We should underline that \eqref{Eq.Gen.Ther.Plate.Eq.} can be reduced to thermoelastic plate equations with Fourier's law of heat conduction if we consider $\ml{A}=(-\Delta)^2$ and $\alpha=1/2$ in \eqref{Eq.Gen.Ther.Plate.Eq.}. But the study of \eqref{Eq.Gen.Ther.Plate.Eq.}, especially asymptotic profiles of solutions, is not a  simple generalization of the study of thermoelastic plate equations. For one thing, due to the fractional power operator $\ml{A}^{\alpha}$ in \eqref{Eq.Gen.Ther.Plate.Eq.}, we need to consider a more refined phase space analysis to derive representation of solutions in the Fourier space. Thus, we will apply WKB analysis associated with multistep diagonalization procedure. For another, concerning the value of parameter $\alpha\in[0,1]$, we expect the exponential stability for large frequencies of solutions for \eqref{Eq.Gen.Ther.Plate.Eq.} will be lost while $\alpha\rightarrow\alpha_0\in[0,1]$, and decay properties of regularity-loss type will come. Here, $\alpha_0$ is a fixed number. One may see \cite{Ueda2018} as an example for a special case of regularity-loss type $\alpha-\beta$ system \eqref{Eq.Alpha-Beta} with $\ml{A}=-\Delta$, $\alpha=0$, $\beta=1$. Nevertheless, up to now, the threshold number $\alpha_0$ of decay properties of regularity-loss type is still not clear from asymptotic profiles point of view. In this paper, we will find this threshold $\alpha_0$ in the consideration of decay properties and asymptotic profiles.

One of the main purpose of this paper is to investigate decay properties and asymptotic profiles of solutions for generalized thermoelastic plate equations \eqref{Eq.Gen.Ther.Plate.Eq.} without any additional constraint condition on initial data. For one thing, concerning decay properties, in the case when $\alpha\in[0,1/3)$ the solutions can be estimated polynomially with regularity-loss, which causes that we need higher regularities of initial data when we derive decay estimates of solutions. However, the effect of regularity-loss will be lost when $\alpha\in[1/3,1]$. Thus, we investigate a threshold $\alpha=1/3$ for decay properties of regularity-loss. This threshold $\alpha=1/3$ for regularity-loss structure is the same as the threshold for polynomial stability of the semigroup to \eqref{Eq.Gen.Ther.Plate.Eq.} investigated in \cite{HaoLiu2013}. Moreover, we also find that the decay rate will be completed changed from $\alpha\in[0,1/2]$ to $\alpha\in(1/2,1]$, which leads to another threshold $\alpha=1/2$. For another,  by investigating some evolution systems, we may derive asymptotic profiles of solutions in a framework of weighted $L^1$ spaces. One may see Theorems \ref{Thm.Asymptotic.Profile.}, \ref{Thm.Asy.Pro.2} and Figure \ref{imgg}. It gives us some opportunities to characterize the long-time asymptotic behavior of solutions.

It is well-known that in the real world, there exists several kinds of resistance in the elongation of a plate. For this reason, plate equations with structural damping in \cite{HorbachIkehataCharao2016,IkehataSoga2015}, and thermoelastic plate equations with structural damping in \cite{Chen2019TEP,FatoriJorgeMaYang2015,Wu2008} are considered. Motivated by these works, in the second step of present paper, we study generalized thermoelastic plate equations \eqref{Eq.Gen.Ther.Plate.Eq.} with additional structural damping term $\ml{A}^{1/2}u_t$, namely,
\begin{equation}\label{Eq.Gen.Ther.Plate.Eq.Add.Damping}
\left\{
\begin{aligned}
&u_{tt}+\ml{A}u-\ml{A}^{\alpha}v+\ml{A}^{1/2}u_t=0,&&x\in\mb{R}^n,\,\,t>0,\\
&v_t+\ml{A}^{\alpha}v+\ml{A}^{\alpha}u_t=0,&&x\in\mb{R}^n,\,\,t>0,\\
&(u,u_t,v)(0,x)=(u_0,u_1,v_0)(x),&&x\in\mb{R}^n,
\end{aligned}
\right.
\end{equation}
where $\alpha\in[0,1]$, and $\ml{A}=(-\Delta)^{\sigma}$ with $\sigma\in[1,\infty)$.  We remark that the model \eqref{Eq.Gen.Ther.Plate.Eq.Add.Damping} is a special case of so-called $\alpha-\beta-\gamma$ system introduced in \cite{FischerRacke}. At this time, we may interpret \eqref{Eq.Gen.Ther.Plate.Eq.Add.Damping} by a doubly dissipative generalized plate equations, where the damping terms consist of structural damping and thermal damping generated by Fourier's law. For this reason, there exists a competition between these  damping effects. The second aim of this paper is to derive decay properties and asymptotic profiles of solutions to \eqref{Eq.Gen.Ther.Plate.Eq.Add.Damping}. Moreover, we will analysis the influence of structural damping on generalized thermoelastic plate equations. In particular, we find that the regularity-loss structure of \eqref{Eq.Gen.Ther.Plate.Eq.} is destroyed by the structural damping term $\ml{A}^{1/2}u_t$.

The paper is organized as follows. We prepare representations of solutions for generalized thermoelastic plate equations in the Fourier space by using multistep of diagonalization and WKB analysis in Section \ref{Subsec.Diag.TPE}. Then, by deriving sharp pointwise estimates in the Fourier space, we obtain $L^2$ decay estimates for \eqref{Eq.Gen.Ther.Plate.Eq.} in Section \ref{Subsec.Decay.Est.TPE}. Moreover, we derive asymptotic profiles of solutions for \eqref{Eq.Gen.Ther.Plate.Eq.} in a framework of weighted $L^1$ space in Section \ref{Subsec.Asym.Prof.TPE}. Then, in Section \ref{Sec.GTPESD}, we discuss decay properties and asymptotic profiles of solutions for generalized thermoelastic plate equations with additional structural damping \eqref{Eq.Gen.Ther.Plate.Eq.Add.Damping}. We next give some applications of our results on thermoelastic plate equations and damped Moore-Gibson-Thompson equation in Section \ref{Sec.Application}. Finally, some concluding remarks complete the paper in Section \ref{Sec.Con.Remark}.
\medskip

\noindent\textbf{Notation.} \quad We give some notations to be used in this paper.  We denote the identity matrix of dimension $k\times k$ by $I_{k\times k}$, the zero matrix of dimension $k\times k$ by $0_{k\times k}$ for $k\in\mb{N}$. On one hand, we write $f\lesssim g$ when there exists a positive constant $C$ such that $f\leqslant Cg$. On the other hand, we write $f\asymp g$ when $g\lesssim f\lesssim g$.

We take the notation $P_f:=\int_{\mb{R}^n}f(x)dx$ to describe integration of the function $f(x)$ over $\mb{R}^n$.

For the diagonal matrix with exponential elements, we define
\begin{align*}
\diag\left(e^{\lambda_j(|\xi|)t}\right)_{j=1}^3:=\diag\left(e^{\lambda_1(|\xi|)t},e^{\lambda_2(|\xi|)t},e^{\lambda_3(|\xi|)t}\right).
\end{align*}

We denote by $H^s$ and $\dot{H}^s$ with $s\geqslant0$,  Bessel and Riesz potential spaces based on $L^2$, respectively. Throughout the paper, $\langle D\rangle^s$ and $|D|^s$ stand for the pseudo-differential operators with symbols $\langle\xi\rangle^s$ and $|\xi|^s$, respectively. Here $\langle\cdot\rangle:=\sqrt{1+|\cdot|^2}$. Next, we define the weighted spaces $L^{1,\kappa}$ with $\kappa\in[0,\infty)$ by
\begin{align*}
L^{1,\kappa}:=\left\{f\in L^1:\|f\|_{L^{1,\kappa}}:=\int_{\mb{R}^n}(1+|x|)^{\kappa}|f(x)|dx<\infty\right\}.
\end{align*}

Finally, let us define the cut-off functions $\chi_{\intt}(\xi),\chi_{\midd}(\xi),\chi_{\extt}(\xi)\in \mathcal{C}^{\infty}$ having their supports in the following small frequency zone, bounded frequency zone and large frequency zone:
\begin{align*}
\ml{Z}_{\intt}(\varepsilon)&:=\left\{\xi\in\mb{R}^n:|\xi|\leqslant\varepsilon\ll1\right\},\\
\ml{Z}_{\midd}(\varepsilon/2,2N)&:=\left\{\xi\in\mb{R}^n:\varepsilon/2\leqslant |\xi|\leqslant 2N\right\},\\
\ml{Z}_{\extt}(N)&:=\left\{\xi\in\mb{R}^n:|\xi|\geqslant N\gg1\right\},
\end{align*}
respectively, so that $\chi_{\intt}(\xi)+\chi_{\midd}(\xi)+\chi_{\extt}(\xi)=1$.


\section{Diagonalization procedure and representations of solutions}\label{Subsec.Diag.TPE}
Let us introduce the quantity $w=w(t,x)$ such that
\begin{align*}
w:=\left(u_t+i(-\Delta)^{\sigma/2}u,u_t-i(-\Delta)^{\sigma/2}u,v\right)^{\mathrm{T}},
\end{align*}
which is the solution to the next first-order system
\begin{equation}\label{Eq.First-Order.TPE}
\left\{
\begin{aligned}
&w_t-B_0(-\Delta)^{\sigma/2}w-B_1(-\Delta)^{\sigma\alpha}w=0,&&x\in\mb{R}^n,\,\,t>0,\\
&w(0,x)=w_0(x),&&x\in\mb{R}^n,
\end{aligned}
\right.
\end{equation}
where the coefficient matrices are given by
\begin{align*}
B_0=\left(
{\begin{array}{*{20}c}
	i & 0 & 0\\
	0 & -i & 0\\
	0 & 0 & 0
	\end{array}}
\right)\quad\text{and}\quad B_1=\left(
{\begin{array}{*{20}c}
	0 & 0 & 1\\
	0 & 0 & 1\\
	-\frac{1}{2} & -\frac{1}{2} & -1
	\end{array}}
\right).
\end{align*}
\begin{remark}\label{Remark.KS.Condition}
	Concerning the special case when $\sigma=1$ and $\alpha=0$ in \eqref{Eq.First-Order.TPE}, we may interpret it as a symmetric hyperbolic-parabolic coupled system with non-symmetric relaxation
	\begin{align*}
	\left\{
	\begin{aligned}
	&w_t-B_0(-\Delta)^{1/2}w-B_1w=0,&&x\in\mb{R}^n,\,\,t>0,\\
	&w(0,x)=w_0(x),&&x\in\mb{R}^n,
	\end{aligned}
	\right.
	\end{align*}
	where $iB_0$ is real symmetric and $B_1$ is non-negative definite but is not real symmetric. Thus, even in the special case when $\sigma=1$ and $\alpha=0$, the general theory on decay properties investigated in \cite{UmedaKawashimaShizuta1984} is not applicable to \eqref{Eq.First-Order.TPE}.
\end{remark}

Before carrying out diagonalization procedure, which was developed in \cite{ReissigWang2005,Yagdjian1997,WangYang2006,YangWang2006,YangWang200602,Jachmann2008,Reissig2016,Chen2019TEP} etc., we employing the partial Fourier transformation with respect to $x$ such that $\hat{w}(t,\xi)=\ml{F}_{x\rightarrow \xi}(w(t,x))$ to \eqref{Eq.First-Order.TPE}, we may derive the first-order system in the Fourier space as follows:
\begin{equation}\label{Eq.Fourier.First-Order.System}
\left\{
\begin{aligned}
&\hat{w}_t-B_0|\xi|^{\sigma}\hat{w}-B_1|\xi|^{2\sigma\alpha}\hat{w}=0,&&\xi\in\mb{R}^n,\,\,t>0,\\
&\hat{w}(0,\xi)=\hat{w}_0(\xi),&&\xi\in\mb{R}^n.
\end{aligned}
\right.
\end{equation}
With the aim of investigating the dominant term in the first step of diagonalization procedure, we now distinguish between the following four cases with respect to the value of $|\xi|$.
\begin{itemize}
	\item Case 2.1: We consider $\alpha\in[0,1/2)$ for $\xi\in \ml{Z}_{\intt}(\varepsilon)$, or $\alpha\in(1/2,1]$ for $\xi\in \ml{Z}_{\extt}(N)$.
	\item Case 2.2: We consider $\alpha\in[0,1/2)$ for $\xi\in \ml{Z}_{\extt}(N)$, or $\alpha\in(1/2,1]$ for $\xi\in \ml{Z}_{\intt}(\varepsilon)$.
	\item Case 2.3: We consider $\alpha=1/2$ for all $\xi\in\mb{R}^n$.
	\item Case 2.4: We consider $\alpha\in[0,1/2)\cup(1/2,1]$ for $\xi\in \ml{Z}_{\midd}(\varepsilon,N)$.
\end{itemize}
To be specific, we will apply the multistep diagonalization procedure to derive asymptotic behavior of eigenvalues in Case 2.1 and Case 2.2, direct diagonalization to get explicit eigenvalues in Case 2.3, and contradiction argument to prove an exponential stability of eigenvalues in Case 2.4.
\subsection{Treatment for Case 2.1}
In this case, we observe that comparing the matrices $B_0|\xi|^{\sigma}$ with $B_1|\xi|^{2\sigma\alpha}$, the matrix $B_1|\xi|^{2\sigma\alpha}$ has a dominant influence. Thus, we now begin diagonalization procedure with the matrix $B_1|\xi|^{2\sigma\alpha}$.

By introducing a new variable $\hat{w}^{(1)}:=N_1^{-1}\hat{w}$ with
\begin{align}\label{Matr.N1}
N_1:=\left(
{\begin{array}{*{20}c}
	-1 & \frac{i\sqrt{3}-1}{2} & \frac{-i\sqrt{3}-1}{2}\\
	1 & \frac{i\sqrt{3}-1}{2} & \frac{-i\sqrt{3}-1}{2}\\
	0 & 1 & 1
	\end{array}}
\right),
\end{align}
we may get
\begin{align*}
\hat{w}^{(1)}_t-\Lambda_1(|\xi|)\hat{w}^{(1)}-N_1^{-1}B_0N_1|\xi|^{\sigma}\hat{w}^{(1)}=0,
\end{align*}
where the diagonal matrix is
\begin{align*}
\Lambda_1(|\xi|):=|\xi|^{2\sigma\alpha}\diag\left(0,-\left(\tfrac{1}{2}+i\tfrac{\sqrt{3}}{2}\right),-\left(\tfrac{1}{2}-i\tfrac{\sqrt{3}}{2}\right)\right).
\end{align*}
Next, we construct a matrix
\begin{align}\label{Matr.N2}
N_2(|\xi|):=|\xi|^{\sigma-2\sigma\alpha}\left(
{\begin{array}{*{20}c}
	0 & -\frac{\sqrt{3}+i}{1+i\sqrt{3}} & \frac{-\sqrt{3}+i}{-1+i\sqrt{3}}\\
	-\frac{2\sqrt{3}}{3(1+i\sqrt{3})} & 0 & 0\\
	\frac{2\sqrt{3}}{3(1-i\sqrt{3})} & 0 & 0
	\end{array}}
\right).
\end{align}
Then, by introducing a new variable $\hat{w}^{(2)}:=(I_{3\times3}+N_2(|\xi|))^{-1}\hat{w}^{(1)}$, we derive 
\begin{align*}
\hat{w}^{(2)}_t-\Lambda_1(|\xi|)\hat{w}^{(2)}-(I_{3\times 3}+N_2(|\xi|))^{-1}\left(N_1^{-1}B_0N_1|\xi|^{\sigma}-[N_2(|\xi|),\Lambda_1(|\xi|)]\right)\hat{w}^{(2)}&\\
-(I_{3\times 3}+N_2(|\xi|))^{-1}N_1^{-1}B_0N_1N_2(|\xi|)|\xi|^{\sigma}\hat{w}^{(2)}&=0.
\end{align*}
Due to the facts that
\begin{align*}
N_1^{-1}B_0N_1|\xi|^{\sigma}-[N_2(|\xi|),\Lambda_1(|\xi|)]=0_{3\times3},
\end{align*}
and
\begin{align*}
(I_{3\times 3}+N_2(|\xi|))^{-1}=I_{3\times 3}-(I_{3\times 3}+N_2(|\xi|))^{-1}N_2(|\xi|),
\end{align*}
the following system holds:
\begin{align*}
\hat{w}^{(2)}_t-\Lambda_1(|\xi|)\hat{w}^{(2)}-N_1^{-1}B_0N_1N_2(|\xi|)|\xi|^{\sigma}\hat{w}^{(2)}-R(|\xi|)\hat{w}^{(2)}=0,
\end{align*}
where we denoted
\begin{align*}
R(|\xi|)=-(I_{3\times 3}+N_2(|\xi|))^{-1}N_2(|\xi|)N_1^{-1}B_0N_1N_2(|\xi|)|\xi|^{\sigma}.
\end{align*}
Similarly, we introduce $\hat{w}^{(3)}:=(I_{3\times3}+N_3(|\xi|))^{-1}\hat{w}^{(2)}$ with
\begin{align}\label{Matr.N3}
N_3(|\xi|):=|\xi|^{2\sigma-4\sigma\alpha}\left(
{\begin{array}{*{20}c}
	0 & 0 & 0\\
	0 & 0 & \frac{-1+i\sqrt{3}}{6}\\
	0 & -\frac{1+i\sqrt{3}}{6} & 0
	\end{array}}
\right)
\end{align}
to obtain directly
\begin{align*}
\hat{w}^{(3)}_t-\Lambda_1(|\xi|)\hat{w}^{(3)}-(I_{3\times 3}+N_3(|\xi|))^{-1}N_1^{-1}B_0N_1N_2(|\xi|)|\xi|^{\sigma}\hat{w}^{(3)}&\\
+(I_{3\times 3}+N_3(|\xi|))^{-1}[N_3(|\xi|),\Lambda_1(|\xi|)]\hat{w}^{(3)}&\\
-(I_{3\times 3}+N_3(|\xi|))^{-1}N_1^{-1}B_0N_1N_2(|\xi|)N_3(|\xi|)|\xi|^{\sigma}\hat{w}^{(3)}&\\
-(I_{3\times 3}+N_3(|\xi|))^{-1}R(|\xi|)(I_{3\times 3}+N_3(|\xi|))\hat{w}^{(3)}&=0.
\end{align*}
According to the calculation
\begin{align*}
\Lambda_2(|\xi|):&=N_1^{-1}B_0N_1N_2(|\xi|)|\xi|^{\sigma}-[N_3(|\xi|),\Lambda_1(|\xi|)]\\
&=|\xi|^{2\sigma-2\sigma\alpha}\diag\left(-1,\tfrac{1}{2}-i\tfrac{\sqrt{3}}{6},\tfrac{1}{2}+i\tfrac{\sqrt{3}}{6}\right),
\end{align*}
we may have
\begin{align*}
\hat{w}^{(3)}_t-(\Lambda_1(|\xi|)+\Lambda_2(|\xi|))\hat{w}^{(3)}-\widetilde{R}(|\xi|)\hat{w}^{(3)}=0,
\end{align*}
where
\begin{align*}
\widetilde{R}(|\xi|)=&-(I_{3\times 3}+N_3(|\xi|))^{-1}\left(N_3(|\xi|)\Lambda_2(|\xi|)+N_1^{-1}B_0N_1N_2(|\xi|)N_3(|\xi|)|\xi|^{\sigma}\right)\\
&+(I_{3\times 3}+N_3(|\xi|))^{-1}R(|\xi|)(I_{3\times 3}+N_3(|\xi|)).
\end{align*}

Summarizing the above diagonalization procedure, we obtain pairwise distinct eigenvalues. Thus, the following proposition for the asymptotic behavior of eigenvalues and representation of solutions can be concluded. For more detail, we refer the reader to \cite{Jachmann2008}.
\begin{prop}\label{Prop.2.1}
	The eigenvalues $\lambda_j=\lambda_j(|\xi|)$ of the coefficient matrix
	\begin{align*}
	B(|\xi|;\sigma,\alpha)=B_0|\xi|^{\sigma}+B_1|\xi|^{2\sigma\alpha}
	\end{align*}
	from \eqref{Eq.Fourier.First-Order.System} behavior if $\alpha\in[0,1/2)$ for $\xi\in\ml{Z}_{\intt}(\varepsilon)$, or $\alpha\in(1/2,1]$ for $\xi\in\ml{Z}_{\extt}(N)$ as
	\begin{align*}
	\lambda_1(|\xi|)&=-|\xi|^{2\sigma-2\sigma\alpha}+\ml{O}\left(|\xi|^{3\sigma-4\sigma\alpha}\right),\\
	\lambda_2(|\xi|)&=-\left(\tfrac{1}{2}+i\tfrac{\sqrt{3}}{2}\right)|\xi|^{2\sigma\alpha}+\left(\tfrac{1}{2}-i\tfrac{\sqrt{3}}{6}\right)|\xi|^{2\sigma-2\sigma\alpha}+\ml{O}\left(|\xi|^{3\sigma-4\sigma\alpha}\right),\\
	\lambda_3(|\xi|)&=-\left(\tfrac{1}{2}-i\tfrac{\sqrt{3}}{2}\right)|\xi|^{2\sigma\alpha}+\left(\tfrac{1}{2}+i\tfrac{\sqrt{3}}{6}\right)|\xi|^{2\sigma-2\sigma\alpha}+\ml{O}\left(|\xi|^{3\sigma-4\sigma\alpha}\right).
	\end{align*}
	Furthermore, the solution to \eqref{Eq.Fourier.First-Order.System} has the representation for $\xi\in\ml{Z}_{\intt}(\varepsilon)$ and $\alpha\in[0,1/2)$ that
	\begin{align*}
	\hat{w}(t,\xi)=N_{\alpha,\intt}(|\xi|)\diag\left(e^{\lambda_j(|\xi|)t}\right)_{j=1}^3N_{\alpha,\intt}^{-1}(|\xi|)\hat{w}_0(\xi),
	\end{align*}
	and the representation for $\xi\in\ml{Z}_{\extt}(N)$ and $\alpha\in(1/2,1]$ that
	\begin{align*}
	\hat{w}(t,\xi)&=N_{\alpha,\extt}(|\xi|)\diag\left(e^{\lambda_j(|\xi|)t}\right)_{j=1}^3N_{\alpha,\extt}^{-1}(|\xi|)\hat{w}_0(\xi),
	\end{align*}
	where
	\begin{align*}
	N_{\alpha,\intt}(|\xi|):&=N_1(I_{3\times 3}+N_2(|\xi|))(I_{3\times 3}+N_3(|\xi|)),&&\text{when }\alpha\in[0,1/2),\\
	N_{\alpha,\extt}(|\xi|):&=N_1(I_{3\times 3}+N_2(|\xi|))(I_{3\times 3}+N_3(|\xi|)),&&\text{when }\alpha\in(1/2,1],
	\end{align*}
	with a constant matrix $N_1$ defined in \eqref{Matr.N1}, $N_2(|\xi|)=\ml{O}\left(|\xi|^{\sigma-2\sigma\alpha}\right)$ defined in \eqref{Matr.N2}, and $N_3(|\xi|)=\ml{O}\left(|\xi|^{2\sigma-4\sigma\alpha}\right)$ defined in \eqref{Matr.N3} for $|\xi|\rightarrow0$ if $\alpha\in[0,1/2)$, and for $|\xi|\rightarrow\infty$ if $\alpha\in(1/2,1]$.
\end{prop}
\subsection{Treatment for Case 2.2}
In this case, we see the matrix $B_0|\xi|^{\sigma}$ has a dominant influence in comparison with the matrix $B_1|\xi|^{\sigma}$ in the first place. Furthermore, it is clear that $B_0$ is diagonal matrix. It means that we do not need to do additional treatment for the matrix $B_0|\xi|^{\sigma}$.

First of all, let us change a new variable such that $\hat{w}^{(1)}:=(I_{3\times 3}+N_4(|\xi|))^{-1}\hat{w}$ with
\begin{align}\label{Matr.N4}
N_4(|\xi|):=|\xi|^{2\sigma\alpha-\sigma}\left(
{\begin{array}{*{20}c}
	0 & 0 & i\\
	0 & 0 & -i\\
	\frac{i}{2} & -\frac{i}{2} & 0
	\end{array}}
\right).
\end{align}
Then, the following system comes:
\begin{align*}
\hat{w}^{(1)}_t-\Lambda_1(|\xi|)\hat{w}^{(1)}-(I_{3\times 3}+N_4(|\xi|))^{-1}\left(B_1|\xi|^{2\sigma\alpha}-[N_4(|\xi|),\Lambda_1(|\xi|)]\right)\hat{w}^{(1)}&\\
-(I_{3\times 3}+N_4(|\xi|))^{-1}B_1N_4(|\xi|)|\xi|^{2\sigma\alpha}\hat{w}^{(1)}&=0,
\end{align*}
where the diagonal matrix is given by
\begin{align*}
\Lambda_1(|\xi|):=B_0|\xi|^{\sigma}=|\xi|^{\sigma}\diag(i,-i,0).
\end{align*}
According to the choice of the matrix $N_4(|\xi|)$, we may calculate
\begin{align*}
\Lambda_2(|\xi|):&=B_1|\xi|^{2\sigma\alpha}-[N_4(|\xi|),\Lambda_1(|\xi|)]\\
&=|\xi|^{2\sigma\alpha}\diag(0,0,-1).
\end{align*}
In other words, it yields
\begin{align*}
\hat{w}^{(1)}_t-(\Lambda_1(|\xi|)+\Lambda_2(|\xi|))\hat{w}^{(1)}-B_2(|\xi|)\hat{w}^{(1)}-R(|\xi|)\hat{w}^{(1)}=0,
\end{align*}
where we defined
\begin{align*}
B_2(|\xi|)&=-N_4(|\xi|)\Lambda_2(|\xi|)+B_1N_4(|\xi|)|\xi|^{2\sigma\alpha},\\
R(|\xi|)&=-(I_{3\times 3}+N_4(|\xi|))^{-1}N_4(|\xi|)B_2(|\xi|).
\end{align*}
Analogously, introducing $\hat{w}^{(2)}:=(I_{3\times 3}+N_5(|\xi|))^{-1}\hat{w}^{(1)}$ with
\begin{align}\label{Matr.N5}
N_5(|\xi|):=|\xi|^{4\sigma\alpha-2\sigma}\left(
{\begin{array}{*{20}c}
	0 & \frac{1}{4} & -1\\
	\frac{1}{4} & 0 & -1\\
	-\frac{1}{2} & -\frac{1}{2} & 0
	\end{array}}
\right),
\end{align}
we may immediately derive
\begin{align*}
\hat{w}^{(2)}_t-(\Lambda_1(|\xi|)+\Lambda_2(|\xi|))\hat{w}^{(2)}-(I_{3\times 3}+N_5(|\xi|))^{-1}\left(B_2(|\xi|)-[N_5(|\xi|),\Lambda_1(|\xi|)]\right)\hat{w}^{(2)}&\\
-(I_{3\times 3}+N_5(|\xi|))^{-1}R(|\xi|)(I_{3\times 3}+N_5(|\xi|))\hat{w}^{(2)}&\\
-(I_{3\times 3}+N_5(|\xi|))^{-1}\left(B_2(|\xi|)N_5(|\xi|)+[\Lambda_2(|\xi|),N_5(|\xi|)]\right)\hat{w}^{(2)}&=0.
\end{align*}
It is obvious that
\begin{align*}
\Lambda_3(|\xi|):&=B_2(|\xi|)-[N_5(|\xi|),\Lambda_1(|\xi|)]\\
&=|\xi|^{4\sigma\alpha-\sigma}\diag\left(\tfrac{i}{2},-\tfrac{i}{2},0\right).
\end{align*}
By denoting 
\begin{align*}
B_3(|\xi|)&=-N_4(|\xi|)B_2(|\xi|)+[\Lambda_2(|\xi|),N_5(|\xi|)],\\
\widetilde{R}(|\xi|)&=-(I_{3\times 3}+N_5(|\xi|))^{-1}\left(N_5(|\xi|)(\Lambda_3(|\xi|)+[\Lambda_2(|\xi|),N_5(|\xi|)])\right)\\
&\quad+(I_{3\times 3}+N_5(|\xi|))^{-1}\left([R(|\xi|),N_5(|\xi|)]+B_2(|\xi|)N_5(|\xi|)\right)\\
&\quad+(I_{3\times 3}+N_4(|\xi|))N_4(|\xi|)N_4(|\xi|)B_2(|\xi|),
\end{align*}
we derive
\begin{align*}
\hat{w}^{(2)}_t-(\Lambda_1(|\xi|)+\Lambda_2(|\xi|)+\Lambda_3(|\xi|))\hat{w}^{(2)}-B_3(|\xi|)\hat{w}^{(2)}-\widetilde{R}(|\xi|)\hat{w}^{(2)}=0.
\end{align*}
Finally, we introduce $\hat{w}^{(3)}:=(I_{3\times 3}+N_6(|\xi|))^{-1}\hat{w}^{(2)}$ with
\begin{align}\label{Matr.N6}
N_6(|\xi|):=|\xi|^{6\sigma\alpha-3\sigma}\left(
{\begin{array}{*{20}c}
	0 & \frac{i}{4} & -i\\
	-\frac{i}{4} & 0 & i\\
	-\frac{i}{2} & \frac{i}{2} & 0
	\end{array}}
\right).
\end{align}
In this way, the next system is derived
\begin{align*}
\hat{w}^{(3)}_t-(\Lambda_1(|\xi|)+\Lambda_2(|\xi|)+\Lambda_3(|\xi|)+\Lambda_4(|\xi|))\hat{w}^{(3)}-\overline{R}(|\xi|)\hat{w}^{(3)}=0,
\end{align*}
where
\begin{align*}
\Lambda_4(|\xi|):&=B_3(|\xi|)-[N_6(|\xi|),\Lambda_1(|\xi|)]\\
&=|\xi|^{6\sigma\alpha-2\sigma}\diag\left(-\tfrac{1}{2},-\tfrac{1}{2},1\right),
\end{align*}
and
\begin{align*}
\overline{R}(|\xi|)=&(I_{3\times 3}+N_6(|\xi|))^{-1}\left([\Lambda_2(|\xi|)+\Lambda_3(|\xi|),N_6(|\xi|)]+B_3(|\xi|)N_6(|\xi|)\right)\\
&+(I_{3\times 3}+N_6(|\xi|))^{-1}\widetilde{R}(|\xi|)(I_{3\times 3}+N_6(|\xi|))-(I_{3\times 3}+N_6(|\xi|))^{-1}N_6(|\xi|)\Lambda_4(|\xi|).
\end{align*}

Considering all steps of diagonalization procedure in the above, we obtain pairwise distinct eigenvalues. Hence, the following proposition for the asymptotic behavior of eigenvalues and representation of solutions can be concluded.
\begin{prop}\label{Prop.2.2}
	The eigenvalues $\lambda_j=\lambda_j(|\xi|)$ of the coefficient matrix
	\begin{align*}
	B(|\xi|;\sigma,\alpha)=B_0|\xi|^{\sigma}+B_1|\xi|^{2\sigma\alpha}
	\end{align*}
	from \eqref{Eq.Fourier.First-Order.System} behavior if $\alpha\in[0,1/2)$ for $\xi\in\ml{Z}_{\extt}(N)$, or $\alpha\in(1/2,1]$ for $\xi\in\ml{Z}_{\intt}(\varepsilon)$ as
	\begin{align*}
	\lambda_{1}(|\xi|)&=i|\xi|^{\sigma}+\tfrac{i}{2}|\xi|^{4\sigma\alpha-\sigma}-\tfrac{1}{2}|\xi|^{6\sigma\alpha-2\sigma}+\ml{O}\left(|\xi|^{8\sigma\alpha-3\sigma}\right),\\
	\lambda_{2}(|\xi|)&=-i|\xi|^{\sigma}-\tfrac{i}{2}|\xi|^{4\sigma\alpha-\sigma}-\tfrac{1}{2}|\xi|^{6\sigma\alpha-2\sigma}+\ml{O}\left(|\xi|^{8\sigma\alpha-3\sigma}\right),\\
	\lambda_{3}(|\xi|)&=-|\xi|^{2\sigma\alpha}+|\xi|^{6\sigma\alpha-2\sigma}+\ml{O}\left(|\xi|^{8\sigma\alpha-3\sigma}\right).
	\end{align*}
	Furthermore, the solution to \eqref{Eq.Fourier.First-Order.System} has the representation for $\xi\in\ml{Z}_{\intt}(\varepsilon)$ and $\alpha\in(1/2,1]$ that
	\begin{align*}
	\hat{w}(t,\xi)&=N_{\alpha,\intt}(|\xi|)\diag\left(e^{\lambda_j(|\xi|)t}\right)_{j=1}^3N_{\alpha,\intt}^{-1}(|\xi|)\hat{w}_0(\xi),
	\end{align*}
	and the representation for $\xi\in\ml{Z}_{\extt}(N)$ and $\alpha\in[0,1/2)$ that
	\begin{align*}
	\hat{w}(t,\xi)&=N_{\alpha,\extt}(|\xi|)\diag\left(e^{\lambda_j(|\xi|)t}\right)_{j=1}^3N_{\alpha,\extt}^{-1}(|\xi|)\hat{w}_0(\xi),
	\end{align*}
	where
	\begin{align*}
	N_{\alpha,\intt}(|\xi|):&=(I_{3\times 3}+N_4(|\xi|))(I_{3\times 3}+N_5(|\xi|))(I_{3\times 3}+N_6(|\xi|)),&&\text{when }\alpha\in(1/2,1],\\
	N_{\alpha,\extt}(|\xi|):&=(I_{3\times 3}+N_4(|\xi|))(I_{3\times 3}+N_5(|\xi|))(I_{3\times 3}+N_6(|\xi|)),&&\text{when }\alpha\in[0,1/2),
	\end{align*}
	with $N_4(|\xi|)=\ml{O}\left(|\xi|^{2\sigma\alpha-\sigma}\right)$ defined in \eqref{Matr.N4}, $N_5(|\xi|)=\ml{O}\left(|\xi|^{4\sigma\alpha-2\sigma}\right)$ defined in \eqref{Matr.N5}, and $N_6(|\xi|)=\ml{O}\left(|\xi|^{8\sigma\alpha-3\sigma}\right)$ defined in \eqref{Matr.N6} for $|\xi|\rightarrow0$ if $\alpha\in(1/2,1]$, and for $|\xi|\rightarrow\infty$ if $\alpha\in[0,1/2)$.
\end{prop}
\subsection{Treatment for Case 2.3}
In this case, we only need to diagonalize the matrix $(B_0+B_1)|\xi|^{\sigma}$ as a whole due to the fact that the matrices $B_0|\xi|^{\sigma}$ and $B_1|\xi|^{2\sigma\alpha}$ have the same influence while $\alpha=1/2$. Then, we have the following result for Case 2.3.
\begin{prop}\label{Prop.2.3}
	Let us consider $\alpha=1/2$ for all $\xi\in \mb{R}^n$. After one step of the diagonalization procedure, the starting Cauchy problem \eqref{Eq.Fourier.First-Order.System} can be transformed to
	\begin{align*}
	\left\{
	\begin{aligned}
	&\hat{w}_t^{(1)}-\Lambda_1(|\xi|)\hat{w}^{(1)}=0,&&\xi\in\mb{R}^n,\,\,t>0,\\
	&\hat{w}^{(1)}(0,\xi)=\hat{w}_0^{(1)}(\xi),&&\xi\in\mb{R}^n,
	\end{aligned}
	\right.
	\end{align*}
	with the diagonal matrix $\Lambda_1(|\xi|)=|\xi|^{\sigma}\diag(y_1,y_2,y_3)$, where the constants $y_1$, $y_2$ and $y_3$ are defined in \eqref{Value y1 y2 y3} and \eqref{Value z1 z2 z3}.
\end{prop}
\begin{proof}
	We directly calculate the eigenvalues of the matrix $(B_0+B_1)|\xi|^{\sigma}$ to get
	\begin{align*}
	0&=\det\left((B_0+B_1)|\xi|^{\sigma}-\lambda I_{3\times 3}\right)\\
	&=\left|
	{\begin{array}{*{20}c}
		i|\xi|^{\sigma}-\lambda & 0 & |\xi|^{\sigma}\\
		0 & -i|\xi|^{\sigma}-\lambda & |\xi|^{\sigma}\\
		-\frac{1}{2}|\xi|^{\sigma} & -\frac{1}{2}|\xi|^{\sigma} & -|\xi|^{\sigma}-\lambda
		\end{array}}
	\right|\\
	&=-\lambda^3-|\xi|^{\sigma}\lambda^2-2|\xi|^{2\sigma}\lambda-|\xi|^{3\sigma}.
	\end{align*}
	Then, the solutions of the above cubic equation are \begin{align*}
	\lambda_j(|\xi|)=|\xi|^{\sigma}y_j\quad\text{with }j=1,2,3,
	\end{align*}
	where
	\begin{equation}\label{Value y1 y2 y3}
	\begin{split}
	y_1=-\tfrac{1}{3}(1+z_1),\quad y_2=-\tfrac{1}{3}\left(1-\tfrac{1}{2}z_1+\tfrac{\sqrt{3}}{2}iz_2\right),\quad y_3=-\tfrac{1}{3}\left(1-\tfrac{1}{2}z_1-\tfrac{\sqrt{3}}{2}iz_2\right).
	\end{split}
	\end{equation}
	Here we denote
	\begin{equation}\label{Value z1 z2 z3}
	\begin{split}
	z_1=\sqrt[3]{\tfrac{1}{2}(3\sqrt{69}+11)}-\sqrt[3]{\tfrac{1}{2}(3\sqrt{69}-11)},\quad z_2=\sqrt[3]{\tfrac{1}{2}(3\sqrt{69}+11)}+\sqrt[3]{\tfrac{1}{2}(3\sqrt{69}-11)}.
	\end{split}
	\end{equation}
	We remark that the constants $y_1\neq y_2\neq y_3$ and $\text{Re }y_j<0$ for all $j=1,2,3$, which means the pairwise distinct eigenvalues with negative real parts are obtained. By introducing $\hat{w}^{(1)}:=T^{-1}\hat{w}$, which satisfies
	\begin{align*}
	\Lambda_1(|\xi|):=|\xi|^{\sigma}T^{-1}(B_0+B_1)T=|\xi|^{\sigma}\diag(y_1,y_2,y_3),
	\end{align*} we complete the proof immediately.
\end{proof}
\subsection{Treatment for Case 2.4}
Finally, in Case 2.4, with the aim of guaranteeing an exponential stability of eigenvalues, we need to derive an exponential decay result by obtaining a priori estimate for eigenvalues for $\xi\in \ml{Z}_{\midd}(\varepsilon,N)$.
\begin{prop}\label{Prop.2.4}
	The solution $\hat{w}(t,\xi)$ to \eqref{Eq.Fourier.First-Order.System} with $\alpha\in[0,1/2)\cup(1/2,1]$ fulfills the following estimate:
	\begin{align*}
	|\hat{w}(t,\xi)|\lesssim e^{-ct}|\hat{w}_0(\xi)|,
	\end{align*}
	for $\xi\in \ml{Z}_{\midd}(\varepsilon,N)$ and $t\geqslant0$ with a positive constant $c$.
\end{prop}
\begin{proof}
	Let us assume that there is a purely imaginary eigenvalue $\lambda=ia$ with $a\in\mb{R}\backslash\{0\}$ of the matrix $B_0|\xi|^{\sigma}+B_1|\xi|^{2\sigma\alpha}$ for $\xi\in \ml{Z}_{\midd}(\varepsilon,N)$. The non-zero real number $a$ fulfills the cubic equation
	\begin{align*}
	0&=\text{det}\left(B_0|\xi|^{\sigma}+B_1|\xi|^{2\sigma\alpha}-\lambda I_{3\times 3}\right)\\
	&=-\lambda^3-|\xi|^{2\sigma\alpha}\lambda^2-\left(|\xi|^{2\sigma}+|\xi|^{4\sigma\alpha}\right)\lambda-|\xi|^{2\sigma+2\sigma\alpha}\\
	&=ia^3+|\xi|^{2\sigma\alpha}a^2-i\left(|\xi|^{2\sigma}+|\xi|^{4\sigma\alpha}\right)a-|\xi|^{2\sigma+2\sigma\alpha},
	\end{align*}
	which implies the following equations:
	\begin{align*}
	\left\{
	\begin{aligned}
	&|\xi|^{2\sigma\alpha}a^2-|\xi|^{2\sigma+2\sigma\alpha}=0,\\
	&ia\left(a^2-|\xi|^{2\sigma}-|\xi|^{4\sigma\alpha}\right)=0.
	\end{aligned}
	\right.
	\end{align*}
	Then, we may obtain the solution of $a^2$ such that
	\begin{align*}
	a^2=|\xi|^{2\sigma}\quad\text{and}\quad a^2=|\xi|^{2\sigma}+|\xi|^{4\sigma\alpha}.
	\end{align*}
	We can conclude a contradiction immediately. In other words, such number $a\in\mb{R}\backslash\{0\}$ not exists. Then, according to the compactness of the bounded zone $\ml{Z}_{\midd}(\varepsilon,N)$ and following the procedure of \cite{ReissigWang2005,Jachmann2008,Chen2019TEP}, we conclude the assertion by employing $\lambda_j(|\xi|)<0$, $j=1,2,3,$ for $|\xi|=\varepsilon$ and $|\xi|=N$ in Propositions \ref{Prop.2.1} and \ref{Prop.2.2}, respectively.
\end{proof}
\section{Decay estimates of solutions}\label{Subsec.Decay.Est.TPE}
In the last section, we have derived asymptotic behavior of eigenvalues in the Fourier space by applying multistep diagonalization procedure. Then, summarizing these results from Propositions \ref{Prop.2.1}, \ref{Prop.2.2}, \ref{Prop.2.3}, and \ref{Prop.2.4}, we may conclude the following sharp pointwise estimate in the Fourier space.
\begin{prop}\label{Prop.PointWise}
	The solution $\hat{w}=\hat{w}(t,\xi)$ to \eqref{Eq.Fourier.First-Order.System} for $\sigma\geqslant1$, $\alpha\in[0,1]$ satisfies the pointwise estimates for any $\xi\in\mb{R}^n$ and $t\geqslant0$
	\begin{align*}
	|\hat{w}(t,\xi)|\lesssim e^{-c\rho(|\xi|)t}|\hat{w}_0(\xi)|,
	\end{align*}
	where the key function $\rho=\rho(|\xi|)$ is defined by
	\begin{align*}
	\rho(|\xi|):=
	\left\{
	\begin{aligned}
	&\frac{|\xi|^{2\sigma-2\sigma\alpha}}{(1+|\xi|^2)^{2\sigma-4\sigma\alpha}},&&\text{if }\alpha\in[0,1/2],\\
	&\frac{|\xi|^{6\sigma\alpha-2\sigma}}{(1+|\xi|^2)^{4\sigma\alpha-2\sigma}},&&\text{if }\alpha\in(1/2,1],
	\end{aligned}
	\right.
	\end{align*}
	with a positive constant $c>0$.
\end{prop}
Let us now give some explanations for decay properties, which are characterized by the pointwise estimates in Proposition \ref{Prop.PointWise}.
\begin{itemize}
	\item We observe that the threshold for decay properties for $|\xi|\rightarrow0$ is $\alpha=1/2$. Precisely, in the case when $\alpha\in[0,1/2]$, the key function fulfills $\rho(|\xi|)\asymp |\xi|^{2\sigma-2\sigma\alpha}$ for $|\xi|\rightarrow0$, and in the case when $\alpha\in(1/2,1]$, the key function satisfies $\rho(|\xi|)\asymp |\xi|^{6\sigma\alpha-2\sigma}$ for $|\xi|\rightarrow0$. Thus, we expect that there exist different decay estimates between these two cases.
	\item The threshold for decay properties for $|\xi|\rightarrow\infty$ is $\alpha=1/3$. In other words, we may see $\rho(|\xi|)\asymp |\xi|^{-2\sigma(1-3\alpha)}$ in the case when $\alpha\in[0,1/3)$, which leads to decay properties of regularity-loss type. However, for the other case when $\alpha\in[1/3,1]$, we expect that the regularity-loss structure is removed.
	\item All in all, the threshold for decay properties can be described by the numbers
	\begin{align*}
	\alpha=\frac{1}{2}\text{ for }|\xi|\rightarrow0\quad\text{and}\quad\alpha=\frac{1}{3}\text{ for }|\xi|\rightarrow\infty.
	\end{align*}
	We should point out that the threshold $\alpha=1/3$ for regularity-loss structure is the same as the threshold for polynomial stability of the semigroup to \eqref{Eq.Gen.Ther.Plate.Eq.} investigated in \cite{HaoLiu2013}.
\end{itemize} 

In order to derive estimates of solutions in a framework of $L^{1,\kappa}$ spaces with $\kappa\in[0,1]$, we first introduce the next lemmas.
\begin{lemma}\label{Lemma.1}
	Let us consider $f\in L^{1,\kappa}$ with $\kappa\in[0,1]$. Then, the following estimates hold:
	\begin{align*}
	\left\|\chi_{\intt}(D)|D|^s\ml{F}_{\xi\rightarrow x}^{-1}\left(e^{-c|\xi|^{\theta}t}\right)f(x)\right\|_{L^2}\lesssim (1+t)^{-\frac{n+2(s+\kappa)}{2\theta}}\|f\|_{L^{1,\kappa}}+(1+t)^{-\frac{n+2s}{2\theta}}|P_f|,
	\end{align*}
	for $s\geqslant0$, $\theta>0$ and $c>0$.
\end{lemma}
\begin{proof}
	Let us employ the Plancherel theorem to get
	\begin{align*}
	&\left\|\chi_{\intt}(D)|D|^s\ml{F}_{\xi\rightarrow x}^{-1}\left(e^{-c|\xi|^{\theta}t}\right)f(x)\right\|_{L^2}=\left\|\chi_{\intt}(\xi)|\xi|^se^{-c|\xi|^{\theta}t}\hat{f}(\xi)\right\|_{L^2}\\
	&\lesssim\left\|\chi_{\intt}(\xi)|\xi|^{s+\kappa}e^{-c|\xi|^{\theta}t}\right\|_{L^2}\|f\|_{L^{1,\kappa}}+\left\|\chi_{\intt}(\xi)|\xi|^{s}e^{-c|\xi|^{\theta}t}\right\|_{L^2}|P_f|,
	\end{align*}
	where we used Lemma 2.1 in \cite{Ikehata2004}.\\
	Due to the fact that
	\begin{align*}
	\left\|\chi_{\intt}(\xi)|\xi|^{s}e^{-c|\xi|^{\theta}t}\right\|_{L^2}\lesssim\left(\int_{|\xi|\leqslant\varepsilon}|\xi|^{2s}e^{-2c|\xi|^{\theta}t}d\xi\right)^{1/2}\lesssim (1+t)^{-\frac{n+2s}{2\theta}},
	\end{align*}
	we may immediately complete the proof.
\end{proof}

\begin{lemma}\label{Lemma.2}
	Let us consider $f\in H^s$ with $s\geqslant0$. Then, the following estimates hold:
	\begin{align*}
	\left\|\chi_{\extt}(D)|D|^{s_0}\ml{F}_{\xi\rightarrow x}^{-1}\left(e^{-c|\xi|^{\theta}t}\right)f(x)\right\|_{L^2}\lesssim
	\left\{
	\begin{aligned}
	&(1+t)^{-\frac{\ell}{-\theta}}\|f\|_{H^{s_0+\ell}}&&\text{if }\theta<0,\\
	&e^{-ct}\|f\|_{H^{s_0}}&&\text{if }\theta\geqslant0,
	\end{aligned}
	\right.
	\end{align*}
	where $0\leqslant s_0+\ell\leqslant s$, $s_0\geqslant0$, $\ell\geqslant0$, $\theta\in\mb{R}$ and $c>0$.
\end{lemma}
\begin{proof}
	For the case when $\theta<0$, we may estimate
	\begin{align*}
	\left\|\chi_{\extt}(D)|D|^{s_0}\ml{F}_{\xi\rightarrow x}^{-1}\left(e^{-c|\xi|^{\theta}t}\right)f(x)\right\|_{L^2}&\lesssim\sup\limits_{|\xi|\geqslant N}\left(|\xi|^{-\ell}e^{-c|\xi|^{-(-\theta)}t}\right)\|f\|_{H^{s_0+\ell}}\\
	&\lesssim (1+t)^{-\frac{\ell}{-\theta}}\|f\|_{H^{s_0+\ell}},
	\end{align*}
	where we used $\|fg\|_{L^2}\lesssim\|g\|_{L^{\infty}}\|f\|_{L^2}$.\\
	For the case when $\theta\geqslant0$, we directly obtain
	\begin{align*}
	\left\|\chi_{\extt}(D)|D|^{s_0}\ml{F}_{\xi\rightarrow x}^{-1}\left(e^{-c|\xi|^{\theta}t}\right)f(x)\right\|_{L^2}&\lesssim\sup\limits_{|\xi|\geqslant N}\left(e^{-c|\xi|^{\theta}t}\right)\|f\|_{H^{s_0}}\\
	&\lesssim e^{-cN^{\theta}t}\|f\|_{H^{s_0}}.
	\end{align*}
	The proof of the lemma is completed.
\end{proof}

Let us state our main result of decay estimates of solutions.
\begin{theorem}\label{Thm.Solution.Esitmate}
	Suppose that initial data $w_0 \in H^s\cap L^{1,\kappa}$  with $s\geqslant0$ and $\kappa\in[0,1]$. Then, the solutions for \eqref{Eq.First-Order.TPE}  with $\sigma\geqslant1$ satisfies the following estimates:
	\begin{align*}
	\|w(t,\cdot)\|_{\dot{H}^{s_0}}\lesssim\left\{
	\begin{aligned}
	&(1+t)^{-\frac{n+2(s_0+\kappa)}{2(2\sigma-2\sigma\alpha)}}\|w_0\|_{L^{1,\kappa}}+(1+t)^{-\frac{\ell}{2\sigma(1-3\alpha)}}\|w_0\|_{H^{s_0+\ell}}\\
	&\qquad\qquad\qquad\qquad\qquad\,\,\,\,\,+(1+t)^{-\frac{n+2s_0}{2(2\sigma-2\sigma\alpha)}}|P_{w_0}|\qquad\text{if }\alpha\in[0,1/3),\\
	&(1+t)^{-\frac{n+2(s_0+\kappa)}{2(2\sigma-2\sigma\alpha)}}\|w_0\|_{L^{1,\kappa}}+e^{-ct}\|w_0\|_{H^{s_0}}\\
	&\qquad\qquad\qquad\qquad\qquad\,\,\,\,\,+(1+t)^{-\frac{n+2s_0}{2(2\sigma-2\sigma\alpha)}}|P_{w_0}|\qquad\text{if }\alpha\in[1/3,1/2],\\
	&(1+t)^{-\frac{n+2(s_0+\kappa)}{2(6\sigma\alpha-2\sigma)}}\|w_0\|_{L^{1,\kappa}}+e^{-ct}\|w_0\|_{H^{s_0}}\\
	&\qquad\qquad\qquad\qquad\qquad\,\,\,\,\,+(1+t)^{-\frac{n+2s_0}{2(6\sigma\alpha-2\sigma)}}|P_{w_0}|\qquad\text{if }\alpha\in(1/2,1],
	\end{aligned}
	\right.
	\end{align*}
	where $0\leqslant s_0+\ell\leqslant s$ with $s_0\geqslant0$ and $\ell\geqslant0$. Here $c>0$ is a positive constant.
\end{theorem}
\begin{proof}
	We apply Proposition \ref{Prop.PointWise} and the Plancherel theorem to get
	\begin{align*}
	\|w(t,\cdot)\|_{\dot{H}^{s_0}}\lesssim&\left\|\chi_{\intt}(D)|D|^{s_0}\ml{F}_{\xi\rightarrow x}^{-1}\left(e^{-c\rho(|\xi|)t}\right)w_0(x)\right\|_{L^2}+e^{-ct}\left\|\chi_{\midd}(D)w_0(x)\right\|_{H^{s_0}}\\
	&+\left\|\chi_{\extt}(D)|D|^{s_0}\ml{F}_{\xi\rightarrow x}^{-1}\left(e^{-c\rho(|\xi|)t}\right)w_0(x)\right\|_{L^2}.
	\end{align*}
	Then, combining with Lemmas \ref{Lemma.1} and \ref{Lemma.2}, we  complete the proof.
\end{proof}
\begin{remark}
	Providing that we consider $|P_{w_0}|\equiv0$ in Theorem \ref{Thm.Solution.Esitmate}, we may observe that decay rate of the corresponding estimates can be improved by $(1+t)^{-\frac{\kappa}{2\sigma-2\sigma\alpha}}$ if $\alpha\in[0,1/2]$, and by $(1+t)^{-\frac{\kappa}{6\sigma\alpha-2\sigma}}$ if $\alpha\in(1/2,1]$, where $\kappa\in[0,1]$.
\end{remark}
\begin{remark}
	We are interested in decay estimates in a framework of weighted $L^1$ spaces. One may also derive $(L^2\cap L^m)-L^2$ estimates with $m\in[1,2]$ by using H\"older's inequality and the Hausdorff-Young inequality.
\end{remark}
\begin{remark}
We distinguish the dominant part of the decay rate in Theorem \ref{Thm.Solution.Esitmate} when $\alpha\in[0,1/3)$ according to the parameter $\ell$. When the parameter $\ell\geqslant0$ fulfills
\begin{align}\label{Cond.Decay.Rate}
\ell<\frac{1-3\alpha}{2(1-\alpha)}(n+2s_0)\quad\text{with }\alpha\in[0,1/3),
\end{align}
the decay rate is determined by $(1+t)^{-\frac{\ell}{2\sigma(1-3\alpha)}}$. Conversely, if the condition \eqref{Cond.Decay.Rate} does not hold, then the decay rate is determined by $(1+t)^{-\frac{n+2s_0}{2(2\sigma-2\sigma\alpha)}}$.
\end{remark}
\begin{remark}
Providing that we choose $\ell=0$ to avoid regularity-loss when $\alpha\in[0,1/3)$, we now only can get a bounded estimate for the case when $\alpha\in[0,1/3)$.
\end{remark}
\begin{remark}
	The decay properties of regularity-loss type appear in Theorem \ref{Thm.Solution.Esitmate} only if $\alpha\in[0,1/3)$ for $|\xi|\rightarrow\infty$. In other words, it requires $s_0+\ell$ regularity for initial data to estimates solutions in the $\dot{H}^{s_0}$ norm. We observe that there occurs the effect of regularity-loss by moving the parameter from $\alpha\in[1/3,1]$ to $\alpha\in[0,1/3)$. In this way, again, we may interpret that the number $\alpha=1/3$ is the threshold for decay properties of regularity-loss type of generalized thermoelastic plate equations \eqref{Eq.Gen.Ther.Plate.Eq.}.
\end{remark}

\section{Asymptotic profiles of solutions}\label{Subsec.Asym.Prof.TPE}
It is well-known that there exist several kinds of idea to investigate asymptotic profiles, for example, \cite{Nishihara2003,Ikehata2014,IkehataMichihisa2019,Chen2019KV}. In this part, we will construct several evolution reference systems with suitably choice of initial data to characterize asymptotic profiles of solutions in a framework of weighted $L^1$ spaces. So, we interpret asymptotic profiles by generalized diffusion phenomena here.

To describe long-time asymptotic behavior of solutions, we investigate asymptotic profiles of solutions for \eqref{Eq.First-Order.TPE} with $\sigma\in[1,\infty)$ and $\alpha\in[0,1/2)\cup(1/2,1]$ in this section. For the case when $\alpha=1/2$, there is not exists any improvement on the decay rate or regularity on the decay estimates in Theorem \ref{Thm.Solution.Esitmate}, since the kernel $\diag\big(e^{-y_j|\xi|^2t}\big)_{j=1}^3$ plays an dominant role in the explicit representation of solution $\hat{w}(t,\xi)$ in Proposition \ref{Prop.2.3}. For this reason, we are interested in the case $\alpha\in[0,1/2)\cup(1/2,1]$. 

We introduce for our further approach some eigenvalues $\tilde{\lambda}_j(|\xi|)$ and $\bar{\lambda}_j(|\xi|)$ with $j=1,2,3$, which are the principal part of the eigenvalues $\lambda_j(|\xi|)$ for small and large frequencies, respectively, in Propositions \ref{Prop.2.1} and \ref{Prop.2.2}. The value of $\tilde{\lambda}_j(|\xi|)$ is given by
\begin{equation*}
\tilde{\lambda}_j(|\xi|):=\left\{
\begin{aligned}
&-|\xi|^{2\sigma-2\sigma\alpha},&&\text{if }j=1,\\
&-\left(\tfrac{1}{2}+i\tfrac{\sqrt{3}}{2}\right)|\xi|^{2\sigma\alpha}+\left(\tfrac{1}{2}-i\tfrac{\sqrt{3}}{6}\right)|\xi|^{2\sigma-2\sigma\alpha},&&\text{if }j=2,\\
&-\left(\tfrac{1}{2}-i\tfrac{\sqrt{3}}{2}\right)|\xi|^{2\sigma\alpha}+\left(\tfrac{1}{2}+i\tfrac{\sqrt{3}}{6}\right)|\xi|^{2\sigma-2\sigma\alpha},&&\text{if }j=3,
\end{aligned}
\right.
\end{equation*}
and the value of $\bar{\lambda}_j(|\xi|)$ is given by
\begin{equation*}
\bar{\lambda}_j(|\xi|):=\left\{
\begin{aligned}
&i|\xi|^{\sigma}+\tfrac{i}{2}|\xi|^{4\sigma\alpha-\sigma}-\tfrac{1}{2}|\xi|^{6\sigma\alpha-2\sigma},&&\text{if }j=1,\\
&-i|\xi|^{\sigma}-\tfrac{i}{2}|\xi|^{4\sigma\alpha-\sigma}-\tfrac{1}{2}|\xi|^{6\sigma\alpha-2\sigma},&&\text{if }j=2,\\
&-|\xi|^{2\sigma\alpha}+|\xi|^{6\sigma\alpha-2\sigma},&&\text{if }j=3.
\end{aligned}
\right.
\end{equation*}

We now introduce the first reference system for $\sigma\geqslant1$ and $\alpha\in[0,1/2)$
\begin{equation}\label{Eq reference system 1}
\left\{
\begin{aligned}
&\widetilde{w}_t+\widetilde{B}_0(-\Delta)^{\sigma\alpha}\widetilde{w}+\widetilde{B}_1(-\Delta)^{\sigma-\sigma\alpha}\widetilde{w}=0,&&x\in\mb{R}^n,\,\,t>0,\\
&\widetilde{w}(0,x)=\ml{F}^{-1}\left((I_{3\times 3}+N_2(|\xi|))^{-1}N_1^{-1}\right)w_0(x),&&x\in\mb{R}^n,
\end{aligned}
\right.
\end{equation}
where the diagonal coefficient matrices are
\begin{align}\label{Mat.tB01}
\widetilde{B}_0=\diag\left(0,\tfrac{1}{2}+i\tfrac{\sqrt{3}}{2},\tfrac{1}{2}-i\tfrac{\sqrt{3}}{2}\right),\quad \widetilde{B}_1=\diag\left(1,-\tfrac{1}{2}+i\tfrac{\sqrt{3}}{6},-\tfrac{1}{2}-i\tfrac{\sqrt{3}}{6}\right).
\end{align}
It is clear that the solution in the Fourier space to \eqref{Eq reference system 1} is given by
\begin{align*}
\ml{F}_{x\rightarrow\xi}(\widetilde{w})(t,\xi)=\diag\left(e^{\tilde{\lambda}_j(|\xi|)t}\right)_{j=1}^3(I_{3\times 3}+N_2(|\xi|))^{-1}N_1^{-1}\hat{w}_0(\xi).
\end{align*}

Moreover, we the second reference system for $\sigma\geqslant1$ and $\alpha\in[0,1/3)\cup(1/2,1]$
\begin{equation}\label{Eq reference system 2}
\left\{
\begin{aligned}
&\overline{w}_t+\overline{B}_0(-\Delta)^{\sigma/2}\overline{w}+\overline{B}_1(-\Delta)^{\sigma\alpha}\overline{w}\\
&\quad+\overline{B}_2(-\Delta)^{2\sigma\alpha-\sigma/2}\overline{w}+\overline{B}_3(-\Delta)^{3\sigma\alpha-\sigma}\overline{w}=0,&&x\in\mb{R}^n,\,\,t>0,\\
&\overline{w}(0,x)=\ml{F}^{-1}\left((I_{3\times 3}+N_5(|\xi|))^{-1}(I_{3\times 3}+N_4(|\xi|))^{-1}\right)w_0(x),&&x\in\mb{R}^n,
\end{aligned}
\right.
\end{equation}
where the diagonal coefficient matrices are
\begin{align}\label{Mat.oB0123}
\overline{B}_0=\diag(-i,i,0),\quad\overline{B}_1=\diag(0,0,1),\quad\overline{B}_2=\diag\left(-\tfrac{i}{2},\tfrac{i}{2},0\right),\quad\overline{B}_3=\diag\left(\tfrac{1}{2},\tfrac{1}{2},-1\right).
\end{align}
Similarly, the solution to \eqref{Eq reference system 2} in the Fourier space is given by
\begin{align*}
\ml{F}_{x\rightarrow\xi}(\overline{w})(t,\xi)=\diag\left(e^{\bar{\lambda}_j(|\xi|)t}\right)_{j=1}^3(I_{3\times 3}+N_5(|\xi|))^{-1}(I_{3\times 3}+N_4(|\xi|))^{-1}\hat{w}_0(\xi).
\end{align*}

Before stating our main results, we denote 
\begin{align*}
S_{\intt}(t,x)&=\chi_{\intt}(D)\ml{F}^{-1}(N_1(I_{3\times 3}+N_2(|\xi|)))\widetilde{w}(t,x),&&\text{for }\alpha\in[0,1/2),\\
\widetilde{S}_{\intt}(t,x)&=\chi_{\intt}(D)\ml{F}^{-1}((I_{3\times 3}+N_4(|\xi|))(I_{3\times 3}+N_5(|\xi|)))\overline{w}(t,x),&&\text{for }\alpha\in(1/2,1],\\
S_{\extt}(t,x)&=\chi_{\extt}(D)\ml{F}^{-1}((I_{3\times 3}+N_4(|\xi|))(I_{3\times 3}+N_5(|\xi|)))\overline{w}(t,x),&&\text{for }\alpha\in[0,1/3).
\end{align*}
\begin{remark}
The functions $S_{\intt}(t,x)$, $\widetilde{S}_{\intt}(t,x)$ and $S_{\extt}(t,x)$ satisfy the estimates
\begin{align*}
\|S_{\intt}(t,\cdot)\|_{\dot{H}^{s_0}}&\lesssim (1+t)^{-\frac{n+2(s_0+\kappa)}{2(2\sigma-2\sigma\alpha)}}\|w_0\|_{L^{1,\kappa}}+(1+t)^{-\frac{n+2s_0}{2(2\sigma-2\sigma\alpha)}}|P_{w_0}|,\\
\|\widetilde{S}_{\intt}(t,\cdot)\|_{\dot{H}^{s_0}}&\lesssim(1+t)^{-\frac{n+2(s_0+\kappa)}{2(6\sigma\alpha-2\sigma)}}\|w_0\|_{L^{1,\kappa}}+(1+t)^{-\frac{n+2s_0}{2(6\sigma\alpha-2\sigma)}}|P_{w_0}|,\\
\|S_{\extt}(t,\cdot)\|_{\dot{H}^{s_0}}&\lesssim (1+t)^{-\frac{\ell}{2\sigma(1-3\alpha)}}\|w_0\|_{H^{s_0+\ell}},
\end{align*}
where initial data $w_0$ is taken from $H^{s}\cap L^{1,\kappa}$ with $0\leqslant s_0+\ell\leqslant s$, $s_0\geqslant0$, $\ell\geqslant0$ and $\kappa\in[0,1]$.
\end{remark}

We now consider the case when $\alpha\in[0,1/3)$ in the first place. Due to Theorem \ref{Thm.Solution.Esitmate}, the decay properties are regularity-loss type. That is to say the decay rate is determined by solutions localized to small frequencies and large frequencies. Moreover, the decay estimates can be obtained by assuming suitable regularity on initial data. For this reason, we have to consider asymptotic profiles of solutions for small frequencies and large frequencies at the same time. 
\begin{theorem}\label{Thm.Asymptotic.Profile.}
	Suppose that initial data $w_0\in H^s\cap L^{1,\kappa}$ with $s\geqslant0$ and $\kappa\in[0,1]$. Then, the following refinement estimates for \eqref{Eq.First-Order.TPE} with $\sigma\geqslant1$ and $\alpha\in[0,1/3)$ hold:
	\begin{align*}
	\|(w-S_{\intt}-S_{\extt})(t,\cdot)\|_{\dot{H}^{s_0}}\lesssim&(1+t)^{-\frac{n+2(s_0+\kappa)}{2(2\sigma-2\sigma\alpha)}-\frac{1-2\alpha}{2(1-\alpha)}}\|w_0\|_{L^{1,\kappa}}\\
	&+(1+t)^{-\frac{\ell}{2\sigma(1-3\alpha)}}\|w_0\|_{H^{s_0+\ell-(\sigma-2\sigma\alpha)}}\\
	&+(1+t)^{-\frac{n+2s_0}{2(2\sigma-2\sigma\alpha)}-\frac{1-2\alpha}{2(1-\alpha)}}|P_{w_0}|,
	\end{align*}
	where $0\leqslant s_0+\ell-(\sigma-2\sigma\alpha)\leqslant s$ with $s_0\geqslant0$ and $\ell\geqslant0$.
\end{theorem}
\begin{proof}
	The proof of the theorem is divided into two parts, including the solution localized to small frequencies and the solution localized to large frequencies.
	
	Let us consider the solution localized to small frequencies first. According to the representation of solution stated in Proposition \ref{Prop.2.1}, the solution localized to small frequencies can be represented by
	\begin{align*}
	\chi_{\intt}(\xi)\hat{w}(t,\xi)=\chi_{\intt}(\xi)N_{\alpha,\intt}(|\xi|)\diag\left(e^{\lambda_j(|\xi|)t}\right)_{j=1}^3N_{\alpha,\intt}^{-1}(|\xi|)\hat{w}_0(\xi),
	\end{align*}
	where
	\begin{align*}
	N_{\alpha,\intt}(|\xi|)=N_1(I_{3\times 3}+N_2(|\xi|))(I_{3\times 3}+N_3(|\xi|)).
	\end{align*}
	By using the relation
	\begin{align*}
	(I_{3\times 3}+N_3(|\xi|))^{-1}=I_{3\times 3}-(I_{3\times 3}+N_3(|\xi|))^{-1}N_3(|\xi|),
	\end{align*}
	we may decompose the solution in the following way:
	\begin{align*}
	\chi_{\intt}(\xi)\hat{w}(t,\xi)=\chi_{\intt}(\xi)(J_{\intt,1}(t,|\xi|)+J_{\intt,2}(t,|\xi|)+J_{\intt,3}(t,|\xi|)),
	\end{align*}
	where
	\begin{align*}
	J_{\intt,1}(t,|\xi|)=&N_1(I_{3\times 3}+N_2(|\xi|))\diag\left(e^{\lambda_j(|\xi|)t}\right)_{j=1}^3\\
	&\times (I_{3\times 3}+N_2(|\xi|))^{-1}N_1^{-1}\hat{w}_0(\xi),\\
	J_{\intt,2}(t,|\xi|)=&N_1(I_{3\times 3}+N_2(|\xi|))N_3(|\xi|)\diag\left(e^{\lambda_j(|\xi|)t}\right)_{j=1}^3\\
	&\times(I_{3\times 3}+N_3(|\xi|))^{-1}(I_{3\times 3}+N_2(|\xi|))^{-1}N_1^{-1}\hat{w}_0(\xi),\\
	J_{\intt,3}(t,|\xi|)=&-N_1(I_{3\times 3}+N_2(|\xi|))\diag\left(e^{\lambda_j(|\xi|)t}\right)_{j=1}^3\\
	&\times(I_{3\times 3}+N_3(|\xi|))^{-1}N_3(|\xi|)(I_{3\times 3}+N_2(|\xi|))^{-1}N_1^{-1}\hat{w}_0(\xi),
	\end{align*}
	where $N_3(|\xi|)=\ml{O}\left(|\xi|^{2\sigma-4\sigma\alpha}\right)$ with $\alpha\in[0,1/3)$.
	
	From the following integral formula:
	\begin{align*}
	e^{\tilde{\lambda}_j(|\xi|)t-g_j(|\xi|)t}-e^{\tilde{\lambda}_j(|\xi|)t}=-g_j(|\xi|)te^{\tilde{\lambda}_j(|\xi|)t}\int_0^1e^{-g_j(|\xi|)t\tau}d\tau,
	\end{align*}
	where $g_j(|\xi|)=\ml{O}\left(|\xi|^{3\sigma-4\sigma\alpha}\right)$, we can estimate
	\begin{align*}
	\left|\chi_{\intt}(\xi)J_{\intt,1}(t,|\xi|)-\hat{S}_{\intt}(t,\xi)\right|\lesssim(1+t)|\xi|^{3\sigma-4\sigma\alpha}e^{-c|\xi|^{2\sigma-2\sigma\alpha}t}|\hat{w}_0(\xi)|,
	\end{align*}
	with a positive constant $c>0$. Thus, by using Lemma \ref{Lemma.1} again
	\begin{align*}
	\left\|\chi_{\intt}(\xi)J_{\intt,1}(t,|\xi|)-\hat{S}_{\intt}(t,\xi)\right\|_{\dot{H}^{s_0}}\lesssim&(1+t)^{1-\frac{n+2(s_0+\kappa+3\sigma-4\sigma\alpha)}{2(2\sigma-2\sigma\alpha)}}\|w_0\|_{L^{1,\kappa}}\\
	&+(1+t)^{1-\frac{n+2(s_0+3\sigma-4\sigma\alpha)}{2(2\sigma-2\sigma\alpha)}}|P_{w_0}|,
	\end{align*}
	and
	\begin{align*}
	\left\|\chi_{\intt}(\xi)(J_{\intt,2}(t,|\xi|)+J_{\intt,3}(t,|\xi|))\right\|_{\dot{H}^{s_0}}\lesssim&(1+t)^{-\frac{n+2(s_0+\kappa+2\sigma-4\sigma\alpha)}{2(2\sigma-2\sigma\alpha)}}\|w_0\|_{L^{1,\kappa}}\\
	&+(1+t)^{-\frac{n+2(s_0+2\sigma-4\sigma\alpha)}{2(2\sigma-2\sigma\alpha)}}|P_{w_0}|.
	\end{align*}
	Summarizing the above estimates, we derive
	\begin{align*}
	\|(\chi_{\intt}(D)w-S_{\intt})(t,\cdot)\|_{\dot{H}^{s_0}(\mb{R}^n)}\lesssim&(1+t)^{-\frac{n+2(s_0+\kappa)}{2(2\sigma-2\sigma\alpha)}-\frac{1-2\alpha}{2(1-\alpha)}}\|w_0\|_{L^{1,\kappa}}\\
	&+(1+t)^{-\frac{n+2s_0}{2(2\sigma-2\sigma\alpha)}-\frac{1-2\alpha}{2(1-\alpha)}}|P_{w_0}|.
	\end{align*}
	
		Let us now consider the solution localized to large frequencies. According to the representation of solution stated in Proposition \ref{Prop.2.2}, the solution localized to large frequencies can be represented by
	\begin{align*}
	\chi_{\extt}(\xi)\hat{w}(t,\xi)=\chi_{\extt}(\xi)N_{\alpha,\extt}(|\xi|)\diag\left(e^{\lambda_j(|\xi|)t}\right)_{j=1}^3N_{\alpha,\extt}^{-1}(|\xi|)\hat{w}_0(\xi),
	\end{align*}
	where
	\begin{align*}
	N_{\alpha,\extt}(|\xi|)=(I_{3\times 3}+N_4(|\xi|))(I_{3\times 3}+N_5(|\xi|))(I_{3\times 3}+N_6(|\xi|)).
	\end{align*}
	By using again
	\begin{align*}
	(I_{3\times 3}+N_6(|\xi|))^{-1}=I_{3\times 3}-(I_{3\times 3}+N_6(|\xi|))^{-1}N_6(|\xi|),
	\end{align*}
	we may decompose the solution by the following form:
	\begin{align*}
	\chi_{\extt}(\xi)\hat{w}(t,\xi)=\chi_{\extt}(\xi)(J_{\extt,1}(t,|\xi|)+J_{\extt,2}(t,|\xi|)+J_{\extt,3}(t,|\xi|)),
	\end{align*}
	where
	\begin{align*}
	J_{\extt,1}(t,|\xi|)=&(I_{3\times 3}+N_4(|\xi|))(I_{3\times 3}+N_5(|\xi|))\diag\left(e^{\lambda_j(|\xi|)t}\right)_{j=1}^3\\
	&\times(I_{3\times 3}+N_5(|\xi|))^{-1}(I_{3\times 3}+N_4(|\xi|))^{-1}\hat{w}_0(\xi),\\
	J_{\extt,2}(t,|\xi|)=&(I_{3\times 3}+N_4(|\xi|))(I_3+N_5(|\xi|))N_6(|\xi|)\diag\left(e^{\lambda_j(|\xi|)t}\right)_{j=1}^3\\
	&\times(I_{3\times 3}+N_6(|\xi|))^{-1}(I_{3\times 3}+N_5(|\xi|))^{-1}(I_{3\times 3}+N_4(|\xi|))^{-1}\hat{w}_0(\xi),\\
	J_{\extt,3}(t,|\xi|)=&-(I_{3\times 3}+N_6(|\xi|))^{-1}+N_4(|\xi|))(I_{3\times 3}+N_5(|\xi|))\diag\left(e^{\lambda_j(|\xi|)t}\right)_{j=1}^3\\
	&\times (I_{3\times 3}+N_6(|\xi|))^{-1}N_6(|\xi|)(I_{3\times 3}+N_5(|\xi|))^{-1}(I_{3\times 3}+N_4(|\xi|))^{-1}\hat{w}_0(\xi),
	\end{align*}
	where $N_6(|\xi|)=\ml{O}\left(|\xi|^{6\sigma\alpha-3\sigma}\right)$ with $\alpha\in[0,1/3)$.
	
	 Similar as the discussion for small frequencies, by applying Lemma \ref{Lemma.2} we obtain
	\begin{align*}
	\left\|\chi_{\extt}(\xi)J_{\extt,1}(t,|\xi|)-\hat{S}_{\extt}(t,\xi)\right\|_{\dot{H}^{s_0}}&\lesssim(1+t)^{-\frac{\ell}{2\sigma(1-3\alpha)}}\|w_0\|_{H^{s_0+\ell-(\sigma-2\sigma\alpha)}},\\
	\left\|\chi_{\extt}(\xi)(J_{\extt,2}(t,|\xi|)+J_{\extt,3}(t,|\xi|))\right\|_{\dot{H}^{s_0}}&\lesssim(1+t)^{-\frac{\ell}{2\sigma(1-3\alpha)}}\|w_0\|_{H^{s_0+\ell-(3\sigma-6\sigma\alpha)}}.
	\end{align*}
	Thus, we deduce
	\begin{align*}
	\|(\chi_{\extt}(D)w-S_{\extt})(t,\cdot)\|_{\dot{H}^{s_0}}\lesssim(1+t)^{-\frac{\ell}{2\sigma(1-3\alpha)}}\|w_0\|_{H^{s_0+\ell-(\sigma-2\sigma\alpha)}}.
	\end{align*}
	
	Finally, an exponential decay for middle frequencies without any require of additional regularity can be obtained such that
	\begin{align*}
	\|\chi_{\midd}(D)w(t,\cdot)\|_{\dot{H}^{s_0}}=\|\chi_{\midd}(\xi)|\xi|^{s_0}\hat{w}(t,\cdot)\|_{L^2}\lesssim e^{-ct}\|w_0\|_{L^2}.
	\end{align*}

	 Summarizing the above derived estimates, we complete the proof.
\end{proof}
\begin{remark}
Let us compare Theorem \ref{Thm.Solution.Esitmate} with Theorem \ref{Thm.Asymptotic.Profile.} in the case when $\alpha\in[0,1/3)$. We observe that the decay rate can be improved by $-\frac{1-2\alpha}{2(1-\alpha)}$ after we subtract the function $S_{\intt}(t,x)$ in the $\dot{H}^{s_0}$ norm. Moreover, the regularity of initial data can be weaken by $\sigma-2\sigma\alpha$ order after we subtract the function $S_{\extt}(t,x)$.
\end{remark}
\begin{remark}
	In Theorem \ref{Thm.Asymptotic.Profile.}, we consider the case 
	\begin{align*}
	0\leqslant s_0+\ell-(\sigma-2\sigma\alpha)\leqslant s.
	\end{align*}
	If the parameters $s_0$, $\ell$, $\sigma$ and $\alpha$ satisfy the relation
	\begin{align*}
	s_0+\ell-(\sigma-2\sigma\alpha)<0,
	\end{align*}
	then we may apply
	\begin{align*}
	\left\|\langle D\rangle^{s_0+\ell-(\sigma-2\sigma\alpha)}w_0\right\|_{L^2}=\left\|\langle\xi\rangle^{s_0+\ell-(\sigma-2\sigma\alpha)}\hat{w}_0(\xi)\right\|_{L^2}\lesssim \|w_0\|_{L^2}.
	\end{align*}
	In other words, comparing with Theorem \ref{Thm.Solution.Esitmate}, we find that the regularity of initial data can be weaken by $s_0+\ell$, in the case when $s_0+\ell<\sigma-2\sigma\alpha$. For example, we consider $\sigma=2$, $\alpha=0$ in \eqref{Eq.Gen.Ther.Plate.Eq.} and $s_0=0$, $\ell=1$ in Theorem \ref{Thm.Asymptotic.Profile.}.
\end{remark}

Next, we start to discuss asymptotic profiles of solutions in the case when $\alpha\in[1/3,1]$. We may observe that in Theorem \ref{Thm.Solution.Esitmate}, when $\alpha\in[1/3,1]$, the decay rate of estimates of solutions in $\dot{H}^{s_0}$ is determined by small frequencies only. For large frequencies, we may derive an exponential decay by assuming initial data taken from $H^{s_0}$ spaces. For this reason, we derive asymptotic profiles of solutions only for small frequencies in the case when $\alpha\in[1/3,1]$.
\begin{theorem}\label{Thm.Asy.Pro.2}
	Suppose that initial data $w_0\in H^s\cap L^{1,\kappa}$ with $s\geqslant0$ and $\kappa\in[0,1]$. For one thing, the following refinement estimates for \eqref{Eq.First-Order.TPE} with $\sigma\geqslant1$ and $\alpha\in[1/3,1/2)$ hold:
	\begin{align*}
	\|(\chi_{\intt}(D)w-S_{\intt})(t,\cdot)\|_{\dot{H}^{s_0}}\lesssim&(1+t)^{-\frac{n+2(s_0+\kappa)}{2(2\sigma-2\sigma\alpha)}-\frac{1-2\alpha}{2(1-\alpha)}}\|w_0\|_{L^{1,\kappa}}\\
	&+(1+t)^{-\frac{n+2s_0}{2(2\sigma-2\sigma\alpha)}-\frac{1-2\alpha}{2(1-\alpha)}}|P_{w_0}|.
	\end{align*}
	For another, the following refinement estimates for \eqref{Eq.First-Order.TPE} with $\sigma\geqslant1$ and $\alpha\in(1/2,1]$ hold:
	\begin{align*}
	\left\|(\chi_{\intt}(D)w-\widetilde{S}_{\intt})(t,\cdot)\right\|_{\dot{H}^{s_0}}\lesssim&(1+t)^{-\frac{n+2(s_0+\kappa)}{2(6\sigma\alpha-2\sigma)}-\frac{2\alpha-1}{2(3\alpha-1)}}\|w_0\|_{L^{1,\kappa}}\\
	&+(1+t)^{-\frac{n+2s_0}{2(6\sigma\alpha-2\sigma)}-\frac{2\alpha-1}{2(3\alpha-1)}}|P_{w_0}|.
	\end{align*}
	Here we assume $0\leqslant s_0\leqslant s$.
\end{theorem}
\begin{proof}
We may follow the proof of Theorem \ref{Thm.Asymptotic.Profile.} for small frequencies and apply Lemma \ref{Lemma.1} to complete the proof immediately.
\end{proof}
\begin{remark}
	Let us compare Theorem \ref{Thm.Solution.Esitmate} with Theorem \ref{Thm.Asy.Pro.2} in the case when $\alpha\in[1/3,1]$. By subtracting the function $S_{\intt}(t,x)$ if $\alpha\in[1/3,1/2)$ and the function $\widetilde{S}_{\intt}(t,x)$ if $\alpha\in(1/2,1]$, we find that the decay rate can be improved by $-\frac{1-2\alpha}{2(1-\alpha)}$ if $\alpha\in[1/3,1/2)$, and $-\frac{2\alpha-1}{2(3\alpha-1)}$ if $\alpha\in(1/2,1]$.
\end{remark}

To end this section, we summarize the results in Theorems \ref{Thm.Asymptotic.Profile.} and \ref{Thm.Asy.Pro.2} to show the asymptotic profiles of solutions, especially, the improvements of decay rate and regularity of initial data.
 \begin{figure}[http]
	\centering
\begin{tikzpicture}
\draw[thick] (-0.2,0) -- (12,0) node[below] {$1$};
\node[left] at (0,-0.2) {{$0$}};
\node[below] at (4,0) {{$1/3$}};
\node[below] at (6,0) {{$1/2$}};
\draw [decorate, decoration={brace, amplitude=5pt}]  (-0.2,0.5) -- (5.8,0.5);
\draw [decorate, decoration={brace, amplitude=5pt}]  (6.2,0.5) -- (12,0.5);
\draw [decorate, decoration={brace, amplitude=5pt}]  (12,-0.5)--(4.2,-0.5) ;
\draw [decorate, decoration={brace, amplitude=5pt}]  (3.8,-0.5)--(-0.2,-0.5) ;
\node[above] at (3,0.6) {decay rate improved $-\frac{1-2\alpha}{2(1-\alpha)}$};
\node[above] at (9,0.6) {decay rate improved $-\frac{2\alpha-1}{2(3\alpha-1)}$};
\node[above] at (2,-1.4) {regularity improved $\sigma-2\sigma\alpha$};
\node[above] at (8,-1.4) {regularity no improvement};
\node[left] at (-0.5,0) {Value of $\alpha$:};
\draw[fill] (-0.2, 0) circle[radius=1pt];
\draw[fill] (4, 0) circle[radius=1pt];
\draw[fill] (6, 0) circle[radius=1pt];
\draw[fill] (12, 0) circle[radius=1pt];
\end{tikzpicture}
	\caption{Improvements of decay rate and regularity of initial data}
\label{imgg}
\end{figure}
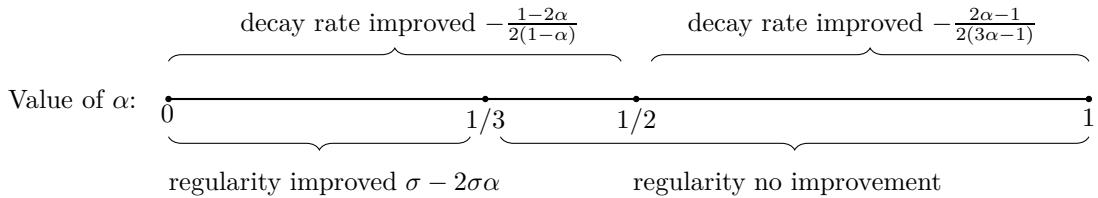

\section{Generalized thermoelastic plate equations with structural damping}\label{Sec.GTPESD}
Throughout this section, we study qualitative properties of solutions, including decay properties and asymptotic profiles in a framework of weighted $L^1$ spaces, for generalized thermoelastic plate equations with additional structural damping \eqref{Eq.Gen.Ther.Plate.Eq.Add.Damping}. Our main purpose in this section is to understand the influence of structural damping term on these qualitative properties.

To do this, we may introduce the quantity $w=w(t,x)$ again such that
\begin{align*}
w:=\left(u_t+i(-\Delta)^{\sigma/2}u,u_t-i(-\Delta)^{\sigma/2}u,v\right)^{\mathrm{T}}.
\end{align*}
Thus, the following first-order system is derived:
\begin{equation}\label{Eq.First-Order.TPE.Add.Damping}
\left\{
\begin{aligned}
&w_t-D_0(-\Delta)^{\sigma/2}w-D_1(-\Delta)^{\sigma\alpha}w=0,&&x\in\mb{R}^n,\,\,t>0,\\
&w(0,x)=w_0(x),&&x\in\mb{R}^n,
\end{aligned}
\right.
\end{equation}
where
\begin{align*}
D_0=\left(
{\begin{array}{*{20}c}
	i-\frac{1}{2} & -\frac{1}{2} & 0\\
	-\frac{1}{2} & -i-\frac{1}{2} & 0\\
	0 & 0 & 0
	\end{array}}
\right)\quad\text{and}\quad D_1=\left(
{\begin{array}{*{20}c}
	0 & 0 & 1\\
	0 & 0 & 1\\
	-\frac{1}{2} & -\frac{1}{2} & -1
	\end{array}}
\right).
\end{align*}
\begin{remark}
Similar as the discussion in Remark \ref{Remark.KS.Condition}, the general theory on decay properties proposed in \cite{UmedaKawashimaShizuta1984} is still not applicable to \eqref{Eq.First-Order.TPE.Add.Damping}, even in the special case when $\sigma=1$ and $\alpha=0$.
\end{remark}
\begin{remark}
Comparing \eqref{Eq.First-Order.TPE} with \eqref{Eq.First-Order.TPE.Add.Damping}, although there exists only small difference between the coefficient matrices $B_0$ and $D_0$, it has great influence on qualitative properties of solutions. We will show the influence later.
\end{remark}

The remaining part of this section is organized as follows. We derive representations of solutions in Subsection \ref{SubSec.Diag.Add.Damp}. Then, by applying these formulas of solutions, decay properties and asymptotic profiles of solutions are obtained in Subsections \ref{SubSec.DecayProp.Add.Damp} and \ref{SubSec.Asym.Prof.Add.Damp}, respectively.

\subsection{Representation of solutions}\label{SubSec.Diag.Add.Damp}
Let us use the partial Fourier transformation to \eqref{Eq.First-Order.TPE.Add.Damping}, we obtain
\begin{equation}\label{Eq.Fourier.First-Order.System.Add}
\left\{
\begin{aligned}
&\hat{w}_t-D_0|\xi|^{\sigma}\hat{w}-D_1|\xi|^{2\sigma\alpha}\hat{w}=0,&&\xi\in\mb{R}^n,\,\,t>0,\\
&\hat{w}(0,\xi)=\hat{w}_0(\xi),&&\xi\in\mb{R}^n.
\end{aligned}
\right.
\end{equation}
Similarly, we now divide the discussion into four cases, concerning about the influence on $|\xi|$.
\begin{itemize}
	\item Case 5.1: We consider $\alpha\in[0,1/2)$ for $\xi\in \ml{Z}_{\intt}(\varepsilon)$, or $\alpha\in(1/2,1]$ for $\xi\in \ml{Z}_{\extt}(N)$.
	\item Case 5.2: We consider $\alpha\in[0,1/2)$ for $\xi\in \ml{Z}_{\extt}(N)$, or $\alpha\in(1/2,1]$ for $\xi\in \ml{Z}_{\intt}(\varepsilon)$.
	\item Case 5.3: We consider $\alpha=1/2$ for all $\xi\in\mb{R}^n$.
	\item Case 5.4: We consider $\alpha\in[0,1/2)\cup(1/2,1]$ for $\xi\in \ml{Z}_{\midd}(\varepsilon,N)$.
\end{itemize}

By using the similar techniques, i.e. diagonalization procedure on Cases 5.1 as well as 5.2, direct diagonalization on Case 5.3, and contradiction argument on Case 5.4, we are able to derive the next four propositions. Hence, we omit the details of the proofs.

\begin{prop}\label{Prop.3.1}
	The eigenvalues $\mu_j=\mu_j(|\xi|)$ of the coefficient matrix
	\begin{align*}
	D(|\xi|;\sigma,\alpha)=D_0|\xi|^{\sigma}+D_1|\xi|^{2\sigma\alpha}
	\end{align*}
	from \eqref{Eq.Fourier.First-Order.System.Add} behavior if $\alpha\in[0,1/2)$ for $\xi\in\ml{Z}_{\intt}(\varepsilon)$, or $\alpha\in(1/2,1]$ for $\xi\in\ml{Z}_{\extt}(N)$ as
	\begin{align*}
	\mu_1(|\xi|)&=-|\xi|^{2\sigma-2\sigma\alpha}+\ml{O}\left(|\xi|^{3\sigma-4\sigma\alpha}\right),\\
	\mu_2(|\xi|)&=-\left(\tfrac{1}{2}+\tfrac{\sqrt{3}}{2}i\right)|\xi|^{2\sigma\alpha}-\left(\tfrac{1}{2}+\tfrac{\sqrt{3}}{6}i\right)|\xi|^{\sigma}+\left(\tfrac{1}{2}-\tfrac{\sqrt{3}}{18}i\right)|\xi|^{2\sigma-2\sigma\alpha}+\ml{O}\left(|\xi|^{3\sigma-4\sigma\alpha}\right),\\
	\mu_3(|\xi|)&=-\left(\tfrac{1}{2}-\tfrac{\sqrt{3}}{2}i\right)|\xi|^{2\sigma\alpha}-\left(\tfrac{1}{2}-\tfrac{\sqrt{3}}{6}i\right)|\xi|^{\sigma}+\left(\tfrac{1}{2}+\tfrac{\sqrt{3}}{18}i\right)|\xi|^{2\sigma-2\sigma\alpha}+\ml{O}\left(|\xi|^{3\sigma-4\sigma\alpha}\right).
	\end{align*}
	Furthermore, the solution to \eqref{Eq.Fourier.First-Order.System.Add} has the representation for $\xi\in\ml{Z}_{\intt}(\varepsilon)$ and $\alpha\in[0,1/2)$ that
	\begin{align*}
	\hat{w}(t,\xi)=M_{\alpha,\intt}(|\xi|)\diag\left(e^{\mu_j(|\xi|)t}\right)_{j=1}^3M_{\alpha,\intt}^{-1}(|\xi|)\hat{w}_0(\xi),
	\end{align*}
	and the representation for $\xi\in\ml{Z}_{\extt}(N)$ and $\alpha\in(1/2,1]$ that
	\begin{align*}
	\hat{w}(t,\xi)&=M_{\alpha,\extt}(|\xi|)\diag\left(e^{\mu_j(|\xi|)t}\right)_{j=1}^3M_{\alpha,\extt}^{-1}(|\xi|)\hat{w}_0(\xi),
	\end{align*}
	where
	\begin{align*}
	M_{\alpha,\intt}(|\xi|):&=M_1(I_{3\times 3}+M_2(|\xi|))(I_{3\times 3}+M_3(|\xi|)),&&\text{when }\alpha\in[0,1/2),\\
	M_{\alpha,\extt}(|\xi|):&=M_1(I_{3\times 3}+M_2(|\xi|))(I_{3\times 3}+M_3(|\xi|)),&&\text{when }\alpha\in(1/2,1],
	\end{align*}
	with the matrices
	\begin{align*}
	M_1:&=\left(
	{\begin{array}{*{20}c}
		-1 & \frac{i\sqrt{3}-1}{2} & \frac{-i\sqrt{3}-1}{2}\\
		1 & \frac{i\sqrt{3}-1}{2} & \frac{-i\sqrt{3}-1}{2}\\
		0 & 1 & 1
		\end{array}}
	\right),\\
	M_2(|\xi|):&=|\xi|^{\sigma-2\sigma\alpha}\left(
	{\begin{array}{*{20}c}
		0 & -\frac{\sqrt{3}+i}{1+\sqrt{3}i} & \frac{\sqrt{3}-i}{1-\sqrt{3}i}\\
		-\frac{2\sqrt{3}}{3(1+\sqrt{3}i)} & 0 & \frac{\sqrt{3}-i}{6i}\\
		\frac{2\sqrt{3}}{3(1-\sqrt{3}i)} & -\frac{\sqrt{3}+i}{6i} & 0
		\end{array}}
	\right),\\
	 M_3(|\xi|):&=|\xi|^{2\sigma-2\sigma\alpha}\left(
	{\begin{array}{*{20}c}
		0 & \frac{2(\sqrt{3}+i)}{3(1+\sqrt{3}i)} & \frac{2(i-\sqrt{3})}{3(1-\sqrt{3}i)}\\
		\frac{2\sqrt{3}}{3(1+\sqrt{3}i)} & 0 & -\frac{2\sqrt{3}+i}{9i}\\
		-\frac{2\sqrt{3}}{3(1-\sqrt{3}i)} & \frac{2\sqrt{3}-i}{9i} & 0
		\end{array}}
	\right).
	\end{align*}
\end{prop}
\begin{prop}\label{Prop.3.2}
	The eigenvalues $\mu_j=\mu_j(|\xi|)$ of the coefficient matrix
	\begin{align*}
	D(|\xi|;\sigma,\alpha)=D_0|\xi|^{\sigma}+D_1|\xi|^{2\sigma\alpha}
	\end{align*}
	from \eqref{Eq.Fourier.First-Order.System.Add} behavior if $\alpha\in[0,1/2)$ for $\xi\in\ml{Z}_{\extt}(N)$, or $\alpha\in(1/2,1]$ for $\xi\in\ml{Z}_{\intt}(\varepsilon)$ as
	\begin{align*}
	\mu_{1}(|\xi|)&=-|\xi|^{2\sigma\alpha}+\ml{O}\left(|\xi|^{4\sigma\alpha-\sigma}\right),\\
	\mu_{2}(|\xi|)&=-\left(\tfrac{1}{2}+\tfrac{\sqrt{3}}{2}i\right)|\xi|^{\sigma}+\ml{O}\left(|\xi|^{4\sigma\alpha-\sigma}\right),\\
	\mu_{3}(|\xi|)&=-\left(\tfrac{1}{2}-\tfrac{\sqrt{3}}{2}i\right)|\xi|^{\sigma}+\ml{O}\left(|\xi|^{4\sigma\alpha-\sigma}\right).
	\end{align*}
	Furthermore, the solution to \eqref{Eq.Fourier.First-Order.System.Add} has the representation for $\xi\in\ml{Z}_{\intt}(\varepsilon)$ and $\alpha\in(1/2,1]$ that
	\begin{align*}
	\hat{w}(t,\xi)&=M_{\alpha,\intt}(|\xi|)\diag\left(e^{\mu_j(|\xi|)t}\right)_{j=1}^3M_{\alpha,\intt}^{-1}(|\xi|)\hat{w}_0(\xi),
	\end{align*}
	and the representation for $\xi\in\ml{Z}_{\extt}(N)$ and $\alpha\in[0,1/2)$ that
	\begin{align*}
	\hat{w}(t,\xi)&=M_{\alpha,\extt}(|\xi|)\diag\left(e^{\mu_j(|\xi|)t}\right)_{j=1}^3M_{\alpha,\extt}^{-1}(|\xi|)\hat{w}_0(\xi),
	\end{align*}
	where
	\begin{align*}
	&M_{\alpha,\intt}(|\xi|):=M_4(I_{3\times 3}+M_5(|\xi|)),&&\text{when }\alpha\in(1/2,1],\\
	&M_{\alpha,\extt}(|\xi|):=M_4(I_{3\times 3}+M_5(|\xi|)),&&\text{when }\alpha\in[0,1/2),
	\end{align*}
	with the matrices
		\begin{align*}
	M_4:&=\left(
	{\begin{array}{*{20}c}
		0 & i(\sqrt{3}-2) & -i(\sqrt{3}+2)\\
		0 & 1 & 1\\
		1 & 0 & 0
		\end{array}}
	\right),\\
	M_5(|\xi|):&=|\xi|^{2\sigma\alpha}\left(
	{\begin{array}{*{20}c}
		0 & \frac{(\sqrt{3}-2)i+1}{\sqrt{3}i+1} & \frac{(\sqrt{3}+2)i-1}{\sqrt{3}i-1}\\
		\frac{(2-i)\sqrt{3}+3}{3(1+\sqrt{3}i)} & 0 & 0\\
		\frac{(-2+i)\sqrt{3}+3}{3(1-\sqrt{3}i)} & 0 & 0
		\end{array}}
	\right).
		\end{align*}
\end{prop}
\begin{prop}\label{Prop.3.3}
	Let us consider $\alpha=1/2$ for all $\xi\in \mb{R}^n$. After one step of the diagonalization procedure, the starting Cauchy problem \eqref{Eq.Fourier.First-Order.System.Add} can be transformed to
	\begin{align*}
	\left\{
	\begin{aligned}
	&\hat{w}_t^{(1)}-\widetilde{\Lambda}_1(|\xi|)\hat{w}^{(1)}=0,&&\xi\in\mb{R}^n,\,\,t>0,\\
	&\hat{w}^{(1)}(0,\xi)=\hat{w}_0^{(1)}(\xi),&&\xi\in\mb{R}^n,
	\end{aligned}
	\right.
	\end{align*}
	with the diagonal matrix $\widetilde{\Lambda}_1(|\xi|)=|\xi|^{\sigma}\diag(-y_4,-y_5,-y_6)$, where the constants $y_4$, $y_5$ and $y_6$ are defined by
	\begin{align*}
	y_4=z_3-\tfrac{5}{9z_3}+\tfrac{2}{3},\,\,\,\,y_5=z_4-\tfrac{5}{9z_4}+\tfrac{2}{3},\,\,\,\,y_6=z_5-\tfrac{5}{9z_5}+\tfrac{2}{3},
	\end{align*}
	where the values of $z_3$, $z_4$ and $z_5$ are given by
	\begin{align*}
	z_3=\tfrac{1}{3}\sqrt[3]{-\tfrac{11}{2}+\tfrac{3}{2}\sqrt{69}},\,\,\,\,z_{4}=\left(-\tfrac{1}{2}+\tfrac{\sqrt{3}}{2}i\right)z_3,\,\,\,\,z_{5}=\left(-\tfrac{1}{2}-\tfrac{\sqrt{3}}{2}i\right)z_3.
	\end{align*}
\end{prop}
\begin{prop}\label{Prop.3.4}
	The solution $\hat{w}(t,\xi)$ to \eqref{Eq.Fourier.First-Order.System.Add} with $\alpha\in[0,1/2)\cup(1/2,1]$ fulfills the following estimate:
	\begin{align*}
	|\hat{w}(t,\xi)|\lesssim e^{-ct}|\hat{w}_0(\xi)|,
	\end{align*}
	for $\xi\in \ml{Z}_{\midd}(\varepsilon,N)$ and $t\geqslant0$ with a positive constant $c$.
\end{prop}
\subsection{Decay estimates of solutions}\label{SubSec.DecayProp.Add.Damp}
Let us summarize the results of Propositions \ref{Prop.3.1}, \ref{Prop.3.2}, \ref{Prop.3.3} and \ref{Prop.3.4} to derive the following sharp pointwise estimates in the Fourier space.
\begin{prop}\label{Prop.PointWise.Add}
	The solution $\hat{w}=\hat{w}(t,\xi)$ to \eqref{Eq.Fourier.First-Order.System.Add} for $\sigma\geqslant1$, $\alpha\in[0,1]$ satisfies the pointwise estimates for any $\xi\in\mb{R}^n$ and $t\geqslant0$
	\begin{align*}
	|\hat{w}(t,\xi)|\lesssim e^{-c\eta(|\xi|)t}|\hat{w}_0(\xi)|,
	\end{align*}
	where the key function $\eta=\eta(|\xi|)$ is defined by
	\begin{align*}
	\eta(|\xi|):=
	\left\{
	\begin{aligned}
	&\frac{|\xi|^{2\sigma-2\sigma\alpha}}{(1+|\xi|^2)^{\sigma-2\sigma\alpha}},&&\text{if }\alpha\in[0,1/2],\\
	&\frac{|\xi|^{2\sigma\alpha}}{(1+|\xi|^2)^{2\sigma\alpha-\sigma}},&&\text{if }\alpha\in(1/2,1],
	\end{aligned}
	\right.
	\end{align*}
	with a positive constant $c>0$.
\end{prop}

We observe that the decay properties of \eqref{Eq.First-Order.TPE} and \eqref{Eq.First-Order.TPE.Add.Damping} are quite different, which is caused by the additional structural damping term. Let us summarize them by the next points.
\begin{itemize}
	\item Concerning the case when $\alpha\in[0,1/2]$, the decay properties for \eqref{Eq.First-Order.TPE.Add.Damping} as $|\xi|\rightarrow0$ can be described by $\eta(|\xi|)\asymp|\xi|^{2\sigma-2\sigma\alpha}$, which is the same as the decay properties for \eqref{Eq.First-Order.TPE}. Nevertheless, for $|\xi|\rightarrow\infty$, the key function $\eta(|\xi|)$ when $\alpha\in[0,1/2]$ asymptotically behaviors as $\eta(|\xi|)\asymp|\xi|^{2\sigma\alpha}$. In other words, the regularity-loss structure is destroyed. We may interpret that this phenomenon is cased by the additional damping term $\ml{A}^{1/2}u_t$. 
	\item Concerning the case when $\alpha\in(1/2,1]$, we find that the decay properties for $|\xi|\rightarrow0$ can be characterized by $\eta(|\xi|)\asymp|\xi|^{2\sigma\alpha}$. It means that \eqref{Eq.First-Order.TPE.Add.Damping} has a stronger decay properties than \eqref{Eq.First-Order.TPE} when $\alpha\in(1/2,1]$. The stronger decay properties will lead to faster decay estimates of solutions. One may see in Theorem \ref{Thm.Solution.Esitmate.Add}.
\end{itemize}

The result on decay estimates of solutions is the following.

\begin{theorem}\label{Thm.Solution.Esitmate.Add}
	Suppose that initial data $w_0 \in H^s\cap L^{1,\kappa}$  with $s\geqslant0$ and $\kappa\in[0,1]$. Then, the solutions for \eqref{Eq.First-Order.TPE.Add.Damping}  with $\sigma\geqslant1$ satisfies the following estimates:
	\begin{align*}
	\|w(t,\cdot)\|_{\dot{H}^{s_0}}\lesssim\left\{
	\begin{aligned}
	&(1+t)^{-\frac{n+2(s_0+\kappa)}{2(2\sigma-2\sigma\alpha)}}\|w_0\|_{L^{1,\kappa}}+e^{-ct}\|w_0\|_{H^{s_0}}\\
	&\qquad\qquad\qquad\qquad\qquad\,\,\,\,\,+(1+t)^{-\frac{n+2s_0}{2(2\sigma-2\sigma\alpha)}}|P_{w_0}|\qquad\text{if }\alpha\in[0,1/2],\\
	&(1+t)^{-\frac{n+2(s_0+\kappa)}{4\sigma\alpha}}\|w_0\|_{L^{1,\kappa}}+e^{-ct}\|w_0\|_{H^{s_0}}\\
	&\qquad\qquad\qquad\qquad\qquad\,\,\,\,\,+(1+t)^{-\frac{n+2s_0}{4\sigma\alpha}}|P_{w_0}|\qquad\quad\,\,\text{if }\alpha\in(1/2,1],
	\end{aligned}
	\right.
	\end{align*}
	where $0\leqslant s_0\leqslant s$ with $s_0\geqslant0$. Here $c>0$ is a positive constant.
\end{theorem}
\begin{proof}
By using sharp pointwise estimates stated in Proposition \ref{Prop.PointWise.Add} and Lemma \ref{Lemma.1}, we may complete the proof immediately.
\end{proof}
\begin{remark}
Providing that we consider $|P_{w_0}|\equiv0$ in Theorem \ref{Thm.Solution.Esitmate.Add}, we may see that decay rate of the corresponding estimates can be improved by $(1+t)^{-\frac{\kappa}{2\sigma-2\sigma\alpha}}$ if $\alpha\in[0,1/2]$ and by $(1+t)^{-\frac{\kappa}{2\sigma\alpha}}$ if $\alpha\in(1/2,1]$, where $\kappa\in[0,1]$.
\end{remark}

Lastly, we discuss the influence of the additional damping term $\ml{A}^{1/2}u_t$ on decay properties for generalized thermoelastic plate equations \eqref{Eq.Gen.Ther.Plate.Eq.}. For one thing, the additional structural damping, in general, can improve decay rate in the decay estimates. However, there exists a competition between structural damping and thermal damping generated by Fourier's law. For the case when $\alpha\in[0,1/2]$, the thermal damping generated by Fourier's law has a dominant influence in comparison with the structural damping. Thus, we only can feel such improvement of decay estimates for the case when $\alpha\in(1/2,1]$.  For another, the structural damping can bring smoothing effect (see for example in coupled systems \cite{Reissig2016,ChenReissig2019,Chen2019TEP}). Up to a point, it can compensate regularity-loss structure. For this reason, the regularity-loss structure for \eqref{Eq.Gen.Ther.Plate.Eq.} is destroyed when $\alpha\in[0,1/3)$.
\subsection{Asymptotic profiles of solutions}\label{SubSec.Asym.Prof.Add.Damp}
With the same reason as Section \ref{Subsec.Asym.Prof.TPE}, we now study asymptotic profiles of solutions for \eqref{Eq.First-Order.TPE.Add.Damping} with $\sigma\in[1,\infty)$ and $\alpha\in[0,1/2)\cup(1/2,1]$. Here, we may interpret asymptotic profiles for small frequencies, due to the fact that the decay rate of estimates of solutions in Theorem \ref{Thm.Solution.Esitmate.Add} is determined by small frequencies only, and exponential decay estimates can be obtained for bounded frequencies and large frequencies.

Let us introduce the reference system for $\sigma\geqslant 1$ and $\alpha\in[0,1/2)$
\begin{equation}\label{Eq reference system 3}
\left\{
\begin{aligned}
&\widetilde{w}_t+\widetilde{D}_0(-\Delta)^{\sigma\alpha}\widetilde{w}+\widetilde{D}_1(-\Delta)^{\sigma}\widetilde{w}+\widetilde{D}_2(-\Delta)^{\sigma-\sigma\alpha}\widetilde{w}=0,&&x\in\mb{R}^n,\,\,t>0,\\
&\widetilde{w}(0,x)=\ml{F}^{-1}\left((I_{3\times 3}+M_2(|\xi|))^{-1}M_1^{-1}\right)w_0(x),&&x\in\mb{R}^n,
\end{aligned}
\right.
\end{equation}
where we denoted
\begin{align*}
\widetilde{D}_0&=\diag\left(0,\tfrac{1}{2}+\tfrac{\sqrt{3}}{2}i,\tfrac{1}{2}-\tfrac{\sqrt{3}}{2}i\right),\\
\widetilde{D}_1&=\diag\left(0,\tfrac{1}{2}+\tfrac{\sqrt{3}}{6}i,\tfrac{1}{2}-\tfrac{\sqrt{3}}{6}i\right),\\
\widetilde{D}_2&=\diag\left(1,-\tfrac{1}{2}+\tfrac{\sqrt{3}}{18}i,-\tfrac{1}{2}-\tfrac{\sqrt{3}}{18}i\right).
\end{align*}
\begin{remark}
	We observe that the new operator $(-\Delta)^{\sigma}$ comes in the reference system \eqref{Eq reference system 3} when $\alpha\in[0,1/2)$. This main difference between \eqref{Eq reference system 3} and \eqref{Eq reference system 1} is caused by the additional structural damping $\ml{A}^{1/2}u_t$. Although from decay properties point of view, the additional damping has no effect on decay rate for $\alpha\in[0,1/2)$, actually, it has great influence on asymptotic profiles of solutions.
\end{remark}
Moreover, we introduce another reference system for $\sigma\geqslant 1$ and $\alpha\in(1/2,1]$
\begin{equation}\label{Eq reference system 4}
\left\{
\begin{aligned}
&\overline{w}_t+\overline{D}_0(-\Delta)^{\sigma/2}\overline{w}+\overline{D}_1(-\Delta)^{\sigma\alpha}\overline{w}=0,&&x\in\mb{R}^n,\,\,t>0,\\
&\overline{w}(0,x)=M_4^{-1}w_0(x),&&x\in\mb{R}^n,
\end{aligned}
\right.
\end{equation}
where we defined\begin{align*}
\overline{D}_0=\diag\left(0,\tfrac{1}{2}+\tfrac{\sqrt{3}}{2}i,\tfrac{1}{2}-\tfrac{\sqrt{3}}{2}i\right),\quad\overline{D}_1=\diag(1,0,0).
\end{align*}

Furthermore, we denote the functions $Q_{\intt}(t,x)$ and $\widetilde{Q}_{\intt}(t,x)$ by
\begin{align*}
Q_{\intt}(t,x)&=\chi_{\intt}(D)\ml{F}^{-1}(M_1(I_{3\times 3}+M_2(|\xi|)))\widetilde{w}(t,x),&&\text{if }\alpha\in[0,1/2),\\
\widetilde{Q}_{\intt}(t,x)&=\chi_{\intt}(D)M_4\overline{w}(t,x),&&\text{if }\alpha\in(1/2,1]&.
\end{align*}
\begin{remark}
The functions $S_{\intt}(t,x)$, $\widetilde{S}_{\intt}(t,x)$ satisfy the estimates
\begin{align*}
	\|Q_{\intt}(t,\cdot)\|_{\dot{H}^{s_0}}&\lesssim(1+t)^{-\frac{n+2(s_0+\kappa)}{2(2\sigma-2\sigma\alpha)}}\|w_0\|_{L^{1,\kappa}}+(1+t)^{-\frac{n+2s_0}{2(2\sigma-2\sigma\alpha)}}|P_{w_0}|,\\
\|\widetilde{Q}_{\intt}(t,\cdot)\|_{\dot{H}^{s_0}}&\lesssim(1+t)^{-\frac{n+2(s_0+\kappa)}{4\sigma\alpha}}\|w_0\|_{L^{1,\kappa}}+(1+t)^{-\frac{n+2s_0}{4\sigma\alpha}}|P_{w_0}|,
\end{align*}
where initial data $w_0$ is taken from $H^{s}\cap L^{1,\kappa}$ with $0\leqslant s_0\leqslant s$ and $\kappa\in[0,1]$.
\end{remark}
\begin{theorem}\label{Thm.Asy.Pro.3}
	Suppose that initial data $w_0\in H^s\cap L^{1,\kappa}$ with $s\geqslant0$ and $\kappa\in[0,1]$. For one thing, the following refinement estimates for \eqref{Eq.First-Order.TPE.Add.Damping} with $\sigma\geqslant1$ and $\alpha\in[0,1/2)$ hold:
	\begin{align*}
	\|(\chi_{\intt}(D)w-Q_{\intt})(t,\cdot)\|_{\dot{H}^{s_0}}\lesssim&(1+t)^{-\frac{n+2(s_0+\kappa)}{2(2\sigma-2\sigma\alpha)}-\frac{1-2\alpha}{2(1-\alpha)}}\|w_0\|_{L^{1,\kappa}}\\
	&+(1+t)^{-\frac{n+2s_0}{2(2\sigma-2\sigma\alpha)}-\frac{1-2\alpha}{2(1-\alpha)}}|P_{w_0}|.
	\end{align*}
	For another, the following refinement estimates for \eqref{Eq.First-Order.TPE.Add.Damping} with $\sigma\geqslant1$ and $\alpha\in(1/2,1]$ hold:
	\begin{align*}
	\left\|(\chi_{\intt}(D)w-\widetilde{Q}_{\intt})(t,\cdot)\right\|_{\dot{H}^{s_0}}\lesssim&(1+t)^{-\frac{n+2(s_0+\kappa)}{4\sigma\alpha}-\frac{2\alpha-1}{2\alpha}}\|w_0\|_{L^{1,\kappa}}\\
	&+(1+t)^{-\frac{n+2s_0}{4\sigma\alpha}-\frac{2\alpha-1}{2\alpha}}|P_{w_0}|.
	\end{align*}
	Here we assume $0\leqslant s_0\leqslant s$.
\end{theorem}
To end this section, we summarize the results in Theorems \ref{Thm.Asymptotic.Profile.} and \ref{Thm.Asy.Pro.2} to show the asymptotic profiles of solutions, especially, the improvements of decay rate and regularity of initial data.
\begin{figure}[http]
	\centering
	\begin{tikzpicture}
	\draw[thick] (-0.2,0) -- (12,0) node[below] {$1$};
	\node[left] at (0,-0.2) {{$0$}};
	\node[below] at (6,0) {{$1/2$}};
	\draw [decorate, decoration={brace, amplitude=5pt}]  (-0.2,0.5) -- (5.8,0.5);
	\draw [decorate, decoration={brace, amplitude=5pt}]  (6.2,0.5) -- (12,0.5);
	\node[above] at (3,0.6) {decay rate improved $-\frac{1-2\alpha}{2(1-\alpha)}$};
	\node[above] at (9,0.6) {decay rate improved $-\frac{2\alpha-1}{2\alpha}$};
	\node[left] at (-0.5,0) {Value of $\alpha$:};
	\draw[fill] (-0.2, 0) circle[radius=1pt];
	\draw[fill] (6, 0) circle[radius=1pt];
	\draw[fill] (12, 0) circle[radius=1pt];
	\end{tikzpicture}
	\caption{Improvements of decay rate}
	\label{imgg1}
\end{figure}
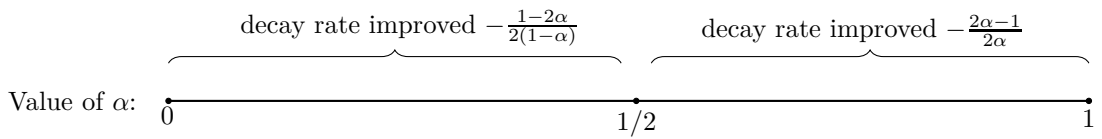
\begin{remark}
Let us discuss the influence of the structural damping $\ml{A}^{1/2}u_t$ on diffusion structure (improvements of decay rate and regularity of initial data). Comparing Theorems \ref{Thm.Asymptotic.Profile.} and \ref{Thm.Asy.Pro.2} with Theorem \ref{Thm.Asy.Pro.3}, or conveniently Figure \ref{imgg} with Figure \ref{imgg1}, we find the following phenomena:
\begin{itemize}
	\item For the case when $\alpha\in[0,1/2)$, there exists the same diffusion structure for decay rate improvement.
	\item For the case when $\alpha\in(1/2,1]$, the improvement for decay rate is shifting from $-\frac{2\alpha-1}{2(3\alpha-1)}$ to $-\frac{2\alpha-1}{2\alpha}$.
	\item For the case when $\alpha\in[0,1/3)$, the improvement for regularity of initial data disappear.
\end{itemize}
\end{remark}

\section{Applications}\label{Sec.Application}
This section is devoted to some applications of our derived results, including decay properties and asymptotic profiles, on thermoelastic plate equations with or without additional structural damping and damped Moore-Gibson-Thompson equation.
\subsection{Applications on thermoelastic plate equations}
For one thing, we consider $\sigma=2$ and $\alpha=1/2$ in \eqref{Eq.Gen.Ther.Plate.Eq.}, which is
\begin{equation}\label{Eq.Thermo.E.P.}
\left\{
\begin{aligned}
&u_{tt}+\Delta^2u+\Delta v=0,&&x\in\mb{R}^n,\,\,t>0,\\
&v_t-\Delta v-\Delta u_t=0,&&x\in\mb{R}^n,\,\,t>0,\\
&(u,u_t,v)(0,x)=(u_0,u_1,v_0)(x),&&x\in\mb{R}^n.
\end{aligned}
\right.
\end{equation}
According to Theorem \ref{Thm.Solution.Esitmate}, we may obtain the following energy estimates:
\begin{align*}
\|(u_t,\Delta u,v)(t,\cdot)\|_{\dot{H}^{s_0}}\lesssim &(1+t)^{-\frac{n+2(s_0+\kappa)}{4}}\|(u_1,\Delta u_0,v_0)\|_{H^{s_0}\cap L^{1,\kappa}}\\
&+(1+t)^{-\frac{n+2s_0}{4}}\left(|P_{u_0}|+|P_{u_1}|+|P_{v_0}|\right),
\end{align*}
for all $s_0\geqslant0$ and $\kappa\in[0,1]$. This result corresponds to the energy estimates derived in Theorem 2.2 of \cite{RackeUeda2016} if $\kappa=0$, $s_0\geqslant0$, and to the energy estimates derived in Theorem 2.1 of \cite{SaidHouari2013} if $\kappa\in[0,1]$, $s_0=0$.	

For another, we consider $\sigma=2$ and $\alpha=1/2$ in \eqref{Eq.Gen.Ther.Plate.Eq.Add.Damping}, which is
\begin{equation}\label{Eq.Thermo.E.P.S.D}
\left\{
\begin{aligned}
&u_{tt}+\Delta^2u+\Delta v-\Delta u_t=0,&&x\in\mb{R}^n,\,\,t>0,\\
&v_t-\Delta v-\Delta u_t=0,&&x\in\mb{R}^n,\,\,t>0,\\
&(u,u_t,v)(0,x)=(u_0,u_1,v_0)(x),&&x\in\mb{R}^n.
\end{aligned}
\right.
\end{equation}
According to Theorem \ref{Thm.Solution.Esitmate.Add}, we may obtain the following energy estimates:
\begin{align*}
\|(u_t,\Delta u,v)(t,\cdot)\|_{\dot{H}^{s_0}}\lesssim &(1+t)^{-\frac{n+2(s_0+\kappa)}{4}}\|(u_1,\Delta u_0,v_0)\|_{H^{s_0}\cap L^{1,\kappa}}\\
&+(1+t)^{-\frac{n+2s_0}{4}}\left(|P_{u_0}|+|P_{u_1}|+|P_{v_0}|\right),
\end{align*}
for all $s_0\geqslant0$ and $\kappa\in[0,1]$. The above estimates corresponds to the energy estimates derived in Theorems 4.2 and 4.3 of \cite{Chen2019TEP}.

\begin{remark}
These two results also tell us that considering thermoelastic plate equations in $\mb{R}^n$, the additional structural damping $-\Delta u_t$ does not exert any influence on energy estimates because now the thermal damping generated by Fourier's law plays a dominant role.
\end{remark}

\subsection{Applications on damped Moore-Gibson-Thompson equation}
Before showing the applications on damped Moore-Gibson-Thompson equation, we should consider the following Moore-Gibson-Thompson equation in the ``critical" case:
\begin{equation}\label{Eq.MGTE}
\left\{
\begin{aligned}
&u_{ttt}+u_{tt}-\Delta u-\Delta u_t=0,&&x\in\mb{R}^n,\,\,t>0,\\
&(u,u_t,u_{tt})(0,x)=(u_0,u_1,u_2)(x),&&x\in\mb{R}^n.
\end{aligned}
\right.
\end{equation}
For this ``critical" case, an energy for Moore-Gibson-Thompson equation is conserved (one may see introduction of \cite{PellicerSaiHouari2017}; the same type of results are derived in \cite{KaltenbacherLasieckaMarchand2011,LasieckaWang2016}). We now prove the energy conservation again by using energy method in the Fourier space.
\begin{prop}
Let us consider an energy such that
\begin{align*}
E_{\text{MGT}}(t):=\frac{1}{2}\|(u_{tt}+u_t)(t,\cdot)\|_{L^2}^2+\frac{1}{2}\||D|(u_t+u)(t,\cdot)\|_{L^2}^2.
\end{align*}
Then, this energy is conserved, i.e., $E_{\text{MGT}}(t)\equiv E_{\text{MGT}}(0)$.
\end{prop}
\begin{proof}
Employing the partial Fourier transform to \eqref{Eq.MGTE}, we get
\begin{equation}\label{Eq.MGTE.F}
\left\{
\begin{aligned}
&\hat{u}_{ttt}+\hat{u}_{tt}+|\xi|^2 \hat{u}+|\xi|^2 \hat{u}_t=0,&&\xi\in\mb{R}^n,\,\,t>0,\\
&(\hat{u},\hat{u}_t,\hat{u}_{tt})(0,\xi)=(\hat{u}_0,\hat{u}_1,\hat{u}_2)(\xi),&&\xi\in\mb{R}^n.
\end{aligned}
\right.
\end{equation}
Then, let us multiply the equation in \eqref{Eq.MGTE.F} by $\bar{\hat{u}}_{tt}$ and  $\bar{\hat{u}}_{t}$, respectively, to get
\begin{align*}
\frac{\partial}{\partial t}\left(\frac{1}{2}|\hat{u}_{tt}|^2+\frac{1}{2}|\xi|^2|\hat{u}_t|^2+|\xi|^2\hat{u}\bar{\hat{u}}_t\right)&=-|\hat{u}_{tt}|^2+|\xi|^2|\hat{u}_t|^2,\\
\frac{\partial}{\partial t}\left(\frac{1}{2}|\hat{u}_{t}|^2+\frac{1}{2}|\xi|^2|\hat{u}|^2+\hat{u}_{tt}\bar{\hat{u}}_t\right)&=|\hat{u}_{tt}|^2-|\xi|^2|\hat{u}_t|^2.
\end{align*}
The sum of the above derived equations shows
\begin{align*}
\frac{1}{2}\left|\hat{u}_{tt}+\hat{u}_t\right|^2+\frac{1}{2}\left||\xi|(\hat{u}_t+\hat{u})\right|^2=\frac{1}{2}\left|\hat{u}_{2}+\hat{u}_1\right|^2+\frac{1}{2}\left||\xi|(\hat{u}_1+\hat{u}_0)\right|^2.
\end{align*}
Finally, by using the Plancherel theorem, we may complete our proof.
\end{proof}

If we now consider an additional damping $u_t$ in the equation of \eqref{Eq.MGTE}, i.e. the following damped Moore-Gibson-Thompson equation in the acoustics theory:
\begin{equation}\label{Eq.DMGTE}
\left\{
\begin{aligned}
&u_{ttt}+u_{tt}-\Delta u-\Delta u_t+u_t=0,&&x\in\mb{R}^n,\,\,t>0,\\
&(u,u_t,u_{tt})(0,x)=(u_0,u_1,u_2)(x),&&x\in\mb{R}^n,
\end{aligned}
\right.
\end{equation}
 then we may expect that there occurs decay estimates for some energies.

To get qualitative properties of \eqref{Eq.DMGTE}, we only need to consider $\alpha=0$ and $\sigma=1$ in \eqref{Eq.Gen.Ther.Plate.Eq.}, that is
\begin{equation}\label{Eq.a=0}
\left\{
\begin{aligned}
&u_{tt}-\Delta u-v=0,&&x\in\mb{R}^n,\,\,t>0,\\
&v_t+ v+ u_t=0,&&x\in\mb{R}^n,\,\,t>0,\\
&(u,u_t,v)(0,x)=(u_0,u_1,v_0)(x),&&x\in\mb{R}^n.
\end{aligned}
\right.
\end{equation}
By direct calculations, we may transfer \eqref{Eq.a=0} to our desired equation \eqref{Eq.DMGTE}, where third initial data is given by $u_2(x):=v_0(x)+\Delta u_0(x)$.

Now, concerning about the Moore-Gibson-Thompson equation with friction \eqref{Eq.DMGTE}, according to Theorem \ref{Thm.Solution.Esitmate}, the solutions fulfill the estimates
\begin{align*}
\|(u_t,|D|u,u_{tt}-\Delta u)(t,\cdot)\|_{\dot{H}^{s_0}}\lesssim &(1+t)^{-\frac{n+2(s_0+\kappa)}{4}}\|(u_1,|D|u_0,u_2-\Delta u_0)(x)\|_{L^{1,\kappa}}\\
&+(1+t)^{-\frac{\ell}{2}}\|(u_1,|D|u_0,u_2-\Delta u_0)(x)\|_{H^{s_0+\ell}}\\
&+(1+t)^{-\frac{n+2s_0}{4}}\left(|P_{u_1}|+|P_{|D|u_0}|+|P_{u_2-\Delta u_0}|\right),
\end{align*}
where $s_0\geqslant0$, $\ell\geqslant0$ and $\kappa\in[0,1]$. Here, we assume
\begin{align*}
\left(u_1,|D|u_0,u_2-\Delta u_0\right)\in H^s\cap L^{1,\kappa},
\end{align*}
for $0\leqslant s_0+\ell\leqslant s$ and $\kappa\in[0,1]$.
\begin{remark}
The above estimate is of regularity-loss type, due to the fact that we obtain the decay rate $(1+t)^{-\frac{\ell}{2}}$ only by assuming the additional $\ell$ order regularity on initial data.
\end{remark}

Moreover, from Theorem \ref{Thm.Asymptotic.Profile.}, the solution 
\begin{align*}
w=\left(u_t+i|D|u,u_t-i|D|u,u_{tt}-\Delta u\right)^{\mathrm{T}},
\end{align*}
has asymptotic profiles as the solution for
\begin{align*}
\widetilde{w}_t-\widetilde{B}_1\Delta\widetilde{w}+\widetilde{B}_0\widetilde{w}=0,
\end{align*}
where the diagonal matrices $\widetilde{B}_j$ for $j=0,1$ are defined in \eqref{Mat.tB01}, and the solution for
\begin{align*}
\Delta\overline{w}_t-\overline{B}_0(-\Delta)^{3/2}\overline{w}+\overline{B}_1\Delta\overline{w}-\overline{B}_2(-\Delta)^{1/2}\overline{w}-\overline{B}_3\overline{w}=0,
\end{align*}
where the diagonal matrices $\overline{B}_j$ for $j=0,1,2,3$ are defined in \eqref{Mat.oB0123}.
\section{Concluding remarks}\label{Sec.Con.Remark}
\begin{remark}[Qualitative properties in other framework]
Throughout this paper, we investigate estimates of solutions and asymptotic profiles in the $L^2$ norm. To study these qualitative properties of solutions in the $L^q$ norm with $q\geqslant2$, we may apply some results proposed in \cite{XuMoriKawashima2015,Chen2019TEP}. More precisely, let us consider a function in Schwartz space $f\in\ml{S}$. Let $\eta=\eta(|\xi|)$ has the following asymptotic behavior:
\begin{equation*}
\eta(|\xi|)\asymp\left\{
\begin{aligned}
&|\xi|^{\kappa_1}&&\text{as }\xi\in\ml{Z}_{\intt}(\varepsilon),\\
&1&&\text{as }\xi\in\ml{Z}_{\midd}(\varepsilon,N),\\
&|\xi|^{\kappa_2}&&\text{as }\xi\in\ml{Z}_{\extt}(N),
\end{aligned}
\right.
\end{equation*}
where $\kappa_1>0$ and $\kappa_2\in\mb{R}$.\\
In the case when $\kappa_2\geqslant0$, the $L^p-L^q$ estimates hold
\begin{align*}
\left\|\ml{F}^{-1}_{\xi\rightarrow x}\left(|\xi|^se^{-\eta(|\xi|)t}\hat{f}(\xi)\right)\right\|_{L^q}\lesssim(1+t)^{-\frac{s}{\kappa_1}-\frac{n}{\kappa_1}\left(\frac{1}{p}-\frac{1}{q}\right)}\|f\|_{L^p}+e^{-ct}\|\langle D\rangle^{s+\ell}f\|_{L^p},
\end{align*}
where $c>0$, $s\geqslant0$, $1\leqslant p\leqslant 2\leqslant q\leqslant\infty$ and $\ell>n(1/p-1/q)$.\\
In the case when $\kappa_2<0$, the $L^p-L^r-L^q$ estimates hold
\begin{align*}
\left\|\ml{F}^{-1}_{\xi\rightarrow x}\left(|\xi|^se^{-\eta(|\xi|)t}\hat{f}(\xi)\right)\right\|_{L^q}\lesssim(1+t)^{-\frac{s}{\kappa_1}-\frac{n}{\kappa_1}\left(\frac{1}{p}-\frac{1}{q}\right)}\|f\|_{L^p}+(1+t)^{\frac{\ell}{\kappa_2}-\frac{n}{\kappa_2}\left(\frac{1}{r}-\frac{1}{q}\right)}\|\langle D\rangle^{s+\ell}f\|_{L^r},
\end{align*}
where $s\geqslant0$, $1\leqslant p,r\leqslant 2\leqslant q\leqslant\infty$ and $\ell>n(1/p-1/q)$.\\
In other words, for the case when $\kappa_2\geqslant0$, we may observe exponential decay for large frequencies if as assume suitable regularity for initial data. However, for the case when $\kappa_2<0$, we only can obtain polynomial decay for large frequencies, even we assume suitable regularity for initial data.
\end{remark}

\begin{remark}[Asymptotic profiles for general $\alpha-\beta$ system] Throughout this paper, by introducing threshold $\alpha=1/2$ and $\alpha=1/3$, we investigate decay properties and asymptotic profiles of solutions to generalized thermoelastic plate equations (or $\alpha-\beta$ system with $\alpha=\beta$). However, it is still open about asymptotic profiles of solutions for the following general $\alpha-\beta$ system:
\begin{equation}\label{Eq.Alpha-Betaz}
\left\{
\begin{aligned}
&u_{tt}+\ml{A}u-\gamma_1\ml{A}^{\alpha}v=0,&&x\in\mb{R}^n,\,\,t>0,\\
&v_t+\gamma_2\ml{A}^{\beta}v+\gamma_1\ml{A}^{\alpha}u_t=0,&&x\in\mb{R}^n,\,\,t>0,\\
&(u,u_t,v)(0,x)=(u_0,u_1,v_0)(x),&&x\in\mb{R}^n,
\end{aligned}
\right.
\end{equation}
where $(\alpha,\beta)\in[0,1]\times[0,1]$ with $\alpha\neq\beta$, $\gamma_1\in\mb{R}\backslash\{0\}$, $\gamma_2\in\mb{R}_+$, and $\ml{A}=(-\Delta)^{\sigma}$ with $\sigma\in[1,\infty)$. We think it is also possible to study asymptotic profiles for \eqref{Eq.Alpha-Betaz}. To derive asymptotic profiles of solutions, we should derive representations of solutions by using asymptotic expansions of eigenvalues/eigenprojections (e.g. \cite{IdeHaramotoKawashima,Ueda2018,Chen2019KV}) or multistep diagonalization procedure (e.g. the method used in Section \ref{Subsec.Diag.TPE}). By using these methods, one should carefully divide the discussion into some cases to analysis the influence from the values of $\alpha,\beta,\sigma$. However, it is not a simple generalization of the known results, and sometimes we need to introduce the constraint condition on initial data (see \cite{Ueda2018} for the case when $\sigma=1$, $\alpha=0$, $\beta=1$ in \eqref{Eq.Alpha-Betaz}).
\end{remark}

\section*{Acknowledgments}
 Yan Liu is supported by the Foundation for natural Science  in Higher Education of Guangdong, China (Grant No. 2018KZDXM048), and the General Project of Science Research of Guangzhou (Grant \# 201707010126). The PhD study of Wenhui Chen is supported by S\"achsiches Landesgraduiertenstipendium.

\bibliographystyle{elsarticle-num}

\end{document}